\documentclass{imanum_mod}
\usepackage{graphicx}
\usepackage{tikz}
\usepackage{algorithm}
\usepackage{algpseudocode}

\usepackage{color}
\definecolor{mygray}{RGB}{47,79,79}

\newcommand\blfootnote[1]{%
  \begingroup
  \renewcommand\thefootnote{}\footnote{#1}%
  \addtocounter{footnote}{-1}%
  \endgroup
}

\graphicspath{ {images/} }
\newcommand{\bolds}[1]{\boldsymbol{#1}}
\newcommand{\sref}[2]{\hyperref[#2]{#1 \ref*{#2}}}
\newcommand{\R}{\mathbf{R}}  

\begin{document}

\title{Finite element approximation of nonlocal fracture models}

\shorttitle{Finite element approximation of nonlocal fracture models}

\author{%
{\sc
Prashant K. Jha\thanks{Email: prashant.j16o@gmail.com, Orcid: https://orcid.org/0000-0003-2158-364X}
and
Robert Lipton\thanks{Corresponding author. Email: lipton@math.lsu.edu, Orcid: https://orcid.org/0000-0002-1382-3204}\\[2pt]
Department of Mathematics, Louisiana State University,\\
Baton Rouge, LA, 70803
}
}
\shortauthorlist{Prashant K. Jha and Robert Lipton}

\maketitle

\blfootnote{This material is based upon work supported by the U. S. Army Research Laboratory and the U. S. Army Research Office under contract/grant number W911NF1610456.}

\begin{abstract}
{We consider nonlocal nonlinear potentials and estimate the rate of convergence of time stepping schemes to the peridynamic equation of motion. 
We begin by establishing the existence of $H^2$ solutions over any finite time interval. Here spatial approximation by finite element interpolations are considered.
The energy stability of the associated semi-discrete time stepping  scheme is established and the approximation of strong and weak formulations of the evolution using  FE interpolations of $H^2$ solutions are investigated. 
The strong and weak form of approximations are shown to converge to the actual solution in the mean square norm at the rate $C_t\Delta t +C_s h^2/\epsilon^2$ where $h$ is the mesh size, $\epsilon$ is the size of nonlocal interaction and $\Delta t$ is the time step. The constants $C_t$ and $C_s$ are independent of $\Delta t$, and  $h$. 
In the absence of nonlinearity a CFL like condition for the energy stability of the central difference time discretization scheme is developed.}
{nonlocal fracture models; peridynamics; cohesive dynamics; numerical analysis; finite element method.}
\end{abstract}


\section{Introduction}
In this article we consider non local models for crack propagation in solids. We focus on the peridynamics formulation introduced in \citet{CMPer-Silling}. The basic idea is to redefine the strain in terms of the difference quotients of the displacement field and allow for nonlocal interaction within some finite horizon. The formulation has a natural length scale given by the size of the horizon. The force at any given material point is computed by considering the deformation of all neighboring material points within a radius given by the size of horizon. 
Here we examine the finite element approximation to the nonlinear nonlocal model proposed and examined in \citet{CMPer-Lipton3,CMPer-Lipton}.  One of the important points of this model is the fact that as the size of the horizon goes to zero, i.e. when we tend to the local limit, the model behaves as if it is an elastic model away from the crack set and has bounded Griffith fracture energy \citet{CMPer-Lipton3,CMPer-Lipton}. Therefore, in the limit, the model not only converges to the local elastic models in regions with small deformation but also has finite Griffith fracture energy associated with a sharp fracture set. The nonlinear potential can be calibrated so that it gives the same fracture toughness as in Linear Elastic Fracture Mechanics models. The force potential used in this model is a smooth version of the prototypical micro elastic brittle model introduced in  \citet{CMPer-Silling}. Further, the slope of the nonlinear force for small strain (as we show in Section \ref{s:intro model}) is controlled by elastic constant of the material. 

The linear peridynamic model converges to the linear elastic model, when the nonlocal length scale goes to zero, this is seen in the convergence of the integral operators to the differential operators, see \citet{CMPer-Emmrich,CMPer-Silling4,AksoyluUnlu}. More fundamentally the convergence of linear peridynamics to local elasticity in the sense of solution operators is shown in \citet{CMPer-Mengesha2}. Crack propagation using the peridynamics model has been considered extensively, see \citet{CMPer-Silling5,BobaruHu,HaBobaru,CMPer-Agwai,CMPer-Ghajari,Diehl,CMPer-Lipton2}. Theoretical analysis of peridynamics models are carried out in \citet{CMPer-Lipton3,CMPer-Lipton,CMPer-Silling7,CMPer-Du3,CMPer-Du5,CMPer-Emmrich,AksoyluParks,CMPer-Dayal2} and \citet{CMPer-Richard,CMPer-Littlewood,CMPer-Du6,CMPer-Gerstle}. Dynamic phase transformations described by  peridynamic theory has been proposed and analyzed in \citet{CMPer-Dayal}. 

In this work, we consider the finite element interpolation given by linear conforming elements. The potential considered in this work is of double well type. One well corresponds to linear elastic deformation and is zero for zero strain and other corresponds to material rupture and has a well at infinity. We consider discrete time stepping methods both in weak and strong form. To proceed we first show the existence of evolutions in $H^2(D;\R^d)\cap H^1_0(D;\R^d)$, see Theorem \ref{thm:existence over finite time domain}. Here the $H^2$ norm for the evolution can become large as the length scale of nonlocal interaction becomes small. 
In Theorem \ref{thm:higher regularity}, we show that the peridynamic evolution can have higher regularity in time when body forces satisfy additional regularity conditions in time. 
We next address the stability of semi-discrete approximation for the nonlinear model and show that the evolution is energetically stable, see Theorem \ref{thm:stab nonlin semi}. We then consider the linearization of the nonlinear model and provide  a stability analysis of the fully discrete approximation. Here we folliw \citet{CMPer-Karaa} and \citet{CMPer-Guan} to obtain a stability condition on $\Delta t$ for the linearized model, see Theorem \ref{thm:cfl condition l}. The rationale is that for small strains the material behaves like a linear elastic material. Recent related work for linear local elastic models establish stability of the general Newmark time discretization \citet{CMPer-Baker,CMPer-Grote,CMPer-Karaa}. This behavior is shown to persist for elastic nonlocal models in \citet{CMPer-Guan}, using techniques in \citet{CMPer-Baker,CMPer-Grote,CMPer-Karaa}.   For the nonlinear model we establish a Lax Richtmyer stability, see Lemma \ref{lemma: stability} and \eqref{LaxRichtmyer}.

The main contribution of this paper is the approximation of $H^2(D;\R^d) \cap H^1_0(D;\R^d)$ peridynamic evolutions by linear conforming elements. The time stepping approximation using finite element interpolation for both strong and weak forms of the evolution problem is shown to converge to the actual solution in the mean square norm at the rate $C_t\Delta t +C_s h^2/\epsilon^2$ where $h$ is the mesh size, $\epsilon$ is the size of nonlocal interaction and $\Delta t$ is the time step, see Theorem \ref{thm:convergence}. The constant $C_t$ is independent of $\Delta t$ and $h$ and depends on the $L^2$ norm of the time derivatives of the exact solution. The constant $C_s$ is also independent of $\Delta t$, and  $h$ and depends on the $H^2$ norm of the exact solution.  We assess the impact of the constants appearing in Theorem \ref{thm:convergence}  for evolution times seen in fracture experiments in  section \ref{s:discussion}. 
Related recent work, \citet{CMPer-JhaLipton}, addresses time stepping methods for the strong form approximation in H\"older space $C^{0,\gamma}(D;\R^d)$ with H\"older exponent $\gamma \in (0,1]$. 
A convergence rate of $C_t\Delta t + C_sh^\gamma/\epsilon^2$ is demonstrated. The constant $C_t$ depends on $L^2$ norm of the time derivative of exact solution and the constant $C_s$ depends on the H\"older norm of the exact solution and the Lipschitz constant of peridynamic force.

The organization of article is as follows: In Section \ref{s:intro model}, we introduce the class of nonlocal nonlinear potentials used in this article. We establish the existence of $H^2(D;\R^d) \cap H^1_0(D;\R^d)$ solutions in Section \ref{s:exist}. In Section \ref{s:fe discr} we describe the finite element approximation and establish energy stability for the semi-discrete in time approximation. In Section \ref{s:central} we consider the central in time discretization and describe the convergence rate of the FEM approximation to the true solution. We establish a CFL like criterion on the time step for stability of the linearized model. 
The proof of claims are given in Section \ref{s:proofs}. We discuss the convergence rate and the associated a-priori error over time scales seen in fracture experiments, see Section \ref{s:discussion}. We provide concluding remarks addressing the existence of asymptotically compatible schemes in the context of fracture, see Section \ref{s:concl}.

\section{Class of bond-based nonlinear potentials}\label{s:intro model}
In this section, we present the nonlinear nonlocal model. Let $D\subset \R^d$, for $d= 2,3$ be the material domain with characteristic length-scale of unity. {To fix ideas $D$ is assumed to be an open set with $C^1$ boundary.} Every material point $x\in D$ interacts nonlocally with all other material points inside a horizon of length $\epsilon\in (0,1)$. Let $H_{\epsilon}(x)$ be the ball of radius $\epsilon$ centered at $x$ containing all points $y$ that interact with $x$. After deformation the material point $x$ assumes position $z = x + u(x)$. In this treatment we assume infinitesimal displacements $u(x)$ so the deformed configuration is the same as the reference configuration and the linearized strain is given by
\begin{align*}
S=S(y,x;u) &= \dfrac{u(y) - u(x)}{|y - x|} \cdot \dfrac{y - x}{|y - x|}.
\end{align*}

We let $t$ denote time and the displacement field $u(t,x)$ evolves according to Newton's second law
\begin{align}\label{eq:per equation}
\rho \partial^2_{tt} u(t,x) &= -\nabla PD^{\epsilon}(u(t))(x) + b(t,x)
\end{align}
for all $x \in D$. Here the body force applied to the domain $D$ can evolve with time and is denoted by $b(t,x)$.  Without loss of generality, we will assume $\rho = 1$. The peridynamic force denoted by $-\nabla PD^\epsilon(u)(x)$ is given by summing up all forces acting on $x$,
\begin{align*}
-\nabla PD^{\epsilon}(u)(x) = \dfrac{2}{\epsilon^d \omega_d} \int_{H_{\epsilon}(x)} \partial_S W^{\epsilon}(S,y - x) \dfrac{y - x}{|y - x|} dy,
\end{align*}
where $\partial_S W^{\epsilon}$ is the force exerted on $x$ by $y$ and is given by the derivative of the nonlocal two point potential $W^{\epsilon}(S,y - x)$ with respect to the strain and $\omega_d$ is volume of unit ball in dimension $d$. 

Let $\partial D$ be the boundary of material domain $D$. 
The Dirichlet boundary condition on $u$ is
\begin{align}\label{eq:per bc}
u(t,x) = 0 \qquad \forall x \in \partial D, \forall t\in [0,T]
\end{align}
and the initial condition is 
\begin{align}\label{eq:per initialvalues}
u(0,x) = u_0(x) \qquad \text{and} \qquad \partial_t u(0,x)=v_0(x).
\end{align}
The initial data and solution $u(t,x)$ are extended by $0$ outside $D$. We remark that traction boundary conditions can be introduced in the nonlocal context by prescribing a body force along a boundary layer of width $\epsilon$ and allowing  the displacement to be free there.

Define the energy $\mathcal{E}^\epsilon(u)(t)$ to be the sum of kinetic and potential energy and is given by
\begin{align}\label{eq:def energy}
\mathcal{E}^\epsilon(u)(t) &= \frac{1}{2} ||\dot{u}(t)||_{L^2} + PD^\epsilon(u(t)).
\end{align}
where potential energy $PD^\epsilon$ is given by
\begin{align*}
PD^{\epsilon}(u) &= \int_D \left[ \dfrac{1}{\epsilon^d \omega_d} \int_{H_\epsilon(x)} |y-x| W^{\epsilon}(S(u), y-x) dy \right] dx.
\end{align*}

\subsection{Nonlocal potential}

We now describe the nonlocal potential. 
We consider potentials $W^{\epsilon}$ of the form
\begin{align}\label{eq:per pot}
W^{\epsilon}(S, y - x) &= \omega(x)\omega(y)\dfrac{J^{\epsilon}(|y - x|)}{\epsilon |y-x|} f(|y - x| S^2),
\end{align}
where $f: \R^{+} \to \R$ is assumed to be positive, smooth, and concave with following properties
\begin{align}\label{eq:per asymptote}
\lim_{r\to 0^+} \dfrac{f(r)}{r} = f'(0), \qquad \lim_{r\to \infty} f(r) = f_{\infty} < \infty .
\end{align}
The peridynamic force $-\nabla PD^\epsilon$ is written as
\begin{align}\label{eq:per force}
-\nabla PD^{\epsilon}(u)(x) = \dfrac{4}{\epsilon^{d+1} \omega_d} \int_{H_{\epsilon}(x)} \omega(x) \omega(y) J^\epsilon(|y - x|) f'(|y - x| S(u)^2) S(u) e_{y - x} dy,
\end{align}
where we used the notation $S(u) = S(y,x;u)$ and $e_{y-x} = \frac{y-x}{|y-x|}$.

The function $J^{\epsilon}(|y - x|)$ models the influence of separation between points $y$ and $x$. Here $J^{\epsilon}(|y-x|) = J(|y-x|/\epsilon)$ can be piecewise smooth  and we define $J$ to be zero outside the ball $\{\xi:\,|\xi|<1\}=H_1(0)$ and $0\leq J(|\xi|) \leq M$ for all $\xi \in H_1(0)$. 

The boundary function $\omega(x)$ is nonnegative and takes the value $1$ for points $x$ inside $D$ of distance $\epsilon$ away from from the boundary $\partial D$. Inside the boundary layer of width $\epsilon$ the function  $\omega(x)$ smoothly decreases from $1$ to $0$ taking the value $0$ on $\partial D$. 

In the sequel we will set
\begin{align}
\bar{\omega}_\xi(x) = \omega(x) \omega(x+ \epsilon \xi)
\end{align}
and we assume
\begin{align*}
|\nabla \bar{\omega}_\xi| \leq C_{\omega_1} < \infty \quad \text{and} \quad |\nabla^2 \bar{\omega}_\xi| \leq C_{\omega_2} < \infty .
\end{align*}

The potential described in (\ref{eq:per pot}) gives the convex-concave dependence, see Fig \ref{fig:per pot}, of $W(S,y - x)$ on the strain $S$ for fixed $y - x$. Initially the deformation is elastic for small strains and then softens as the strain becomes larger. The critical strain where the force between $x$ and $y$ begins to soften is given by $S_c(y, x) := \bar{r}/\sqrt{|y - x|}$ and the force decreases monotonically for
\begin{align}
|S(y, x;u)| > S_c.
\end{align}
Here $\bar{r}$ is the inflection point of $r \mapsto f(r^2)$ and is the root of following equation
\begin{align}
f'({r}^2) + 2{r}^2 f''({r}^2) = 0.
\end{align}

\begin{figure}
\centering
\includegraphics[scale=0.3]{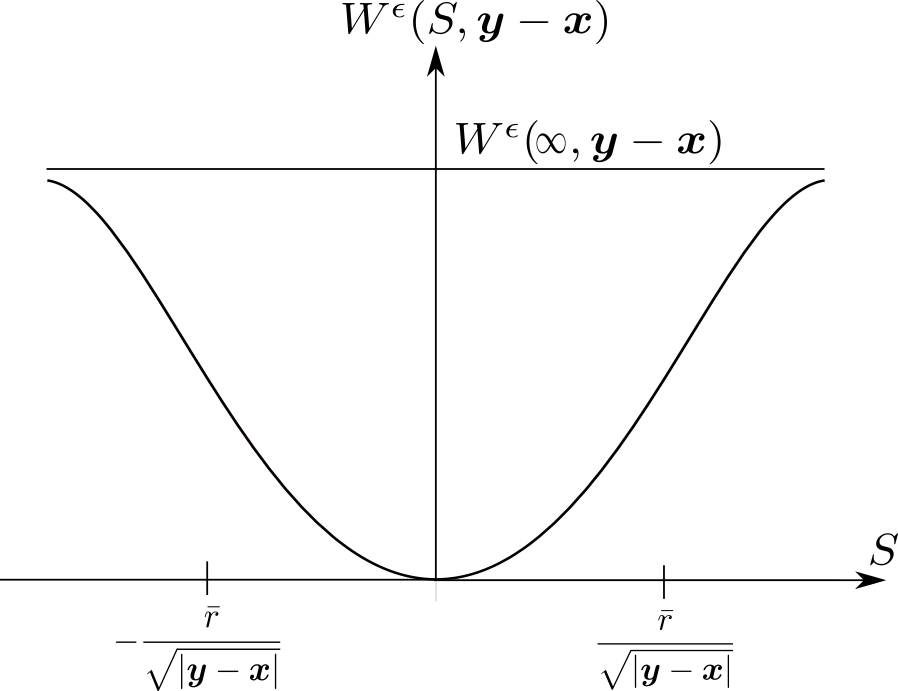}
\caption{Two-point potential $W^\epsilon(S,y - x)$ as a function of strain $S$ for fixed $y - x$.}
 \label{fig:per pot}
\end{figure} 

\begin{figure}
\centering
\includegraphics[scale=0.3]{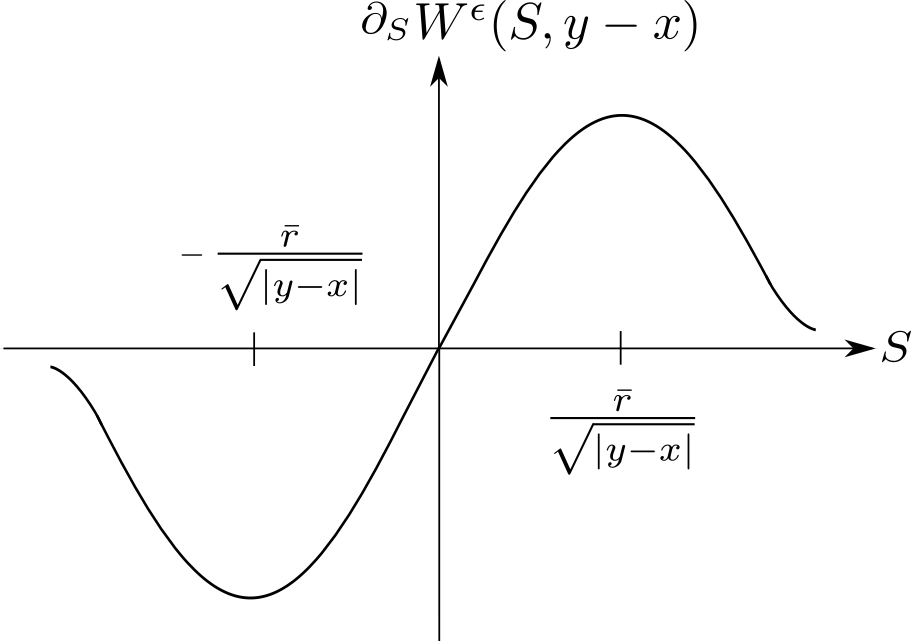}
\caption{Nonlocal force $\partial_S W^\epsilon(S,y - x)$ as a function of strain $S$ for fixed $y - x$. Second derivative of $W^\epsilon(S,y-x)$ is zero at $\pm \bar{r}/\sqrt{|y -x|}$.}
 \label{fig:first der per pot}
\end{figure}  

In [Theorem 5.2, \citet{CMPer-Lipton}], it is shown that in the limit $\epsilon \to 0$, the peridynamics solution has bounded  linear elastic fracture energy, provided the initial data $(u_0,v_0)$ has bounded linear elastic fracture energy and $u_0$ is bounded. The elastic constant (Lam\'e constant $\lambda$ and $\mu$) and energy release rate of the limiting energy is given by
\begin{align*}
\lambda = \mu = C_d f'(0) M_d, \quad \mathcal{G}_c = \dfrac{2 \omega_{d-1}}{\omega_d} f_\infty M_d
\end{align*}
where $ M_d = \int_0^1 J(r) r^d dr$ and $f_\infty = \lim_{r\to \infty} f(r)$. $C_d= 2/3, 1/4, 1/5$ for $d=1,2,3$ respectively and $\omega_n = 1, 2, \pi, 4\pi/3$ for $n=0,1,2,3$. Therefore, $f'(0)$ and $f_\infty$ are determined by the Lam\'e constant $\lambda$ and fracture toughness $\mathcal{G}_c$.

\subsection{Weak formulation}
We now give the weak formulation of the evolution.  Here the notation $||\cdot||$ denotes the  $L^2$ norm on $D$, $||\cdot||_\infty$ is used for the $L^\infty$ norm on $D$, and $||\cdot||_n$ for Sobolev $H^n$ norms on $D$ for $n=1,2$. We denote the dot product in $L^2(D;\R^d)$ by $(\cdot, \cdot)$. For $n=1,2$ we recall that the Sobolev space $H^n_0(D;\R^d)$ is given by the closure in the $H^n$ norm of the  functions that are infinitely differentiable with compact support in $D$. Multiplying (\ref{eq:per equation}) by a test function in $\tilde{u}$ in $H^1_0(D;\R^d)$  we get
\begin{align*}
(\ddot{u}(t), \tilde{u}) &= (-\nabla PD^\epsilon(u(t)), \tilde{u}) + (b(t), \tilde{u}).
\end{align*}
An integration by parts easily shows for all $u,v$ in $H^1_0(D;\R^d)$ that
\begin{align*}
(-\nabla PD^\epsilon(u), v) &= -a^\epsilon(u,v),
\end{align*}
where
\begin{align}
a^\epsilon(u, v) &= \dfrac{2}{\epsilon^{d+1} \omega_d} \int_D \int_{H_\epsilon(x)} \omega(x) \omega(y) J^\epsilon(|y - x|) f'(|y - x| S(u)^2) |y - x| S(u) S(v) dy dx. \label{eq:operator a}
\end{align}
The same identities hold for all functions in $L^2(D;\R^d)$.

It is readily verified that the peridynamic force and energy are bounded for all functions in $L^2(D;\R^d)$. Here the bound on the force follows from the Lipschitz property of the force in $L^2(D;\R^d)$, see, \eqref{eq:lipschitz const ltwo}.  The peridynamic force is also bounded for functions $u$ in $H^1(D;\R^d)$. This again follows  from the Lipschitz property of the force in $H^1(D;\R^d)$ using arguments established in section \ref{s:proofs}.  The boundedness of the energy $PD^\epsilon(u)$  in both $L^2(D;\R^d)$ and $H^1(D;\R^d)$ follows from the boundedness of the bond potential  energy ${W}^\epsilon(S(u),y-x)$ used in the definition of $PD^\epsilon(u)$, see \eqref{eq:per pot}. More generally this also shows that $PD^\epsilon(u)<\infty$ for $u\in L^1(D;\R^d)$.

Finally, the weak form of the evolution in terms of the operator $a^\epsilon$ becomes
\begin{align}
(\ddot{u}(t), \tilde{u}) + a^\epsilon(u(t), \tilde{u}) &= (b(t), \tilde{u}). \label{eq:weak form}
\end{align}

Using definition of $a^\epsilon$ in (\ref{eq:operator a}), one easily sees that
\begin{align}
\dfrac{d}{dt} \mathcal{E}^\epsilon(u)(t) = (\ddot{u}(t), \dot{u}(t)) + a^\epsilon(u(t), \dot{u}(t))\label{eq:energy relat}.
\end{align}


\section{Existence of solutions in $H^2 \cap H^1_0$}\label{s:exist}
We consider function space $W$ given by
\begin{align}
W := H^2(D;\R^d) \cap H^1_0(D;\R^d) = \{ v \in H^2(D;\R^d): \gamma v = 0, \: \text{on } \partial D\},
\end{align}
where $\gamma v$ is the trace of function $v$ on the boundary $\partial D$. Norm on $W$ is $H^2(D;\R^d)$ norm. In this section, we show that for suitable initial condition and body force solutions of peridynamic equation exist in $W$. We will assume that $u\in W$ is extended by zero outside $D$.

We note the following Sobolev embedding properties of $H^2(D;\R^d)$ when $D$ is a $C^1$ domain. 
\begin{itemize}
\item From Theorem 2.72 of \citet{MAFa-Demengel}, there exists a constant $C_{e_1}$ independent of $u \in H^2(D;\R^d)$ such that 
\begin{align}
||u||_{L^\infty(D;\R^d)} \leq C_{e_1} ||u||_{H^2(D;\R^d)} . \label{eq:sob embedd 1}
\end{align}

\item Further application of standard embedding theorems (e.g., Theorem 2.72 of \citet{MAFa-Demengel}), shows there exists a constant $C_{e_2}$ independent of $u$ such that
\begin{align}
||\nabla u||_{L^q(D;\R^{d\times d})} \leq C_{e_2} ||\nabla u||_{H^1(D;\R^{d\times d})} \leq C_{e_2} ||u||_{H^2(D;\R^d)}, \label{eq:sob embedd 2}
\end{align}
for any $q$ such that $2\leq q< \infty$ when $d=2$ and $2\leq q \leq 6$ when $d=3$.
\end{itemize}

In what follows, we  first state the Lipschitz continuity property for  $-\nabla PD^\epsilon(u)$. We then state the existence theorem for solutions over finite time intervals. These are proved in section \ref{s:proofs}.  We now write the peridynamic evolution equation as an equivalent first order system with $y_1(t)=u(t)$ and $y_2(t)=v(t)$ with $v(t)=\dot{u}(t)$. Let $y = (y_1, y_2)^T$ where $y_1,y_2 \in W$ and let $F^{\epsilon}(y,t) = (F^{\epsilon}_1(y,t), F^{\epsilon}_2(y,t))^T$ such that
\begin{align}
F^\epsilon_1(y,t) &:= y_2, \label{eq:per first order eqn 1} \\
F^\epsilon_2(y, t) &:= -\nabla PD^\epsilon(y_1) + b(t). \label{eq:per first order eqn 2}
\end{align}
The initial boundary value is equivalent to the initial boundary value problem for the first order system given by
\begin{align}\label{eq:per first order}
\dot{y}(t) = F^{\epsilon}(y,t),
\end{align}
with initial condition given by $y(0) = (u_0, v_0)^T \in W\times W$. Recall that we denote the norm on $H^2(D;\R^d)$ as $||\cdot||_2$.

\begin{theorem}\label{thm:lip in H2}
\textbf{Lipschitz continuity of the peridynamic force}\\
For any $u,v\in W$, we have
\begin{align}
&||-\nabla PD^\epsilon(u) - (-\nabla PD^\epsilon(v))||_2 \notag \\
&\leq \dfrac{\bar{L}_1 + \bar{L}_2(||u||_2 + ||v||_2) + \bar{L}_3(||u||_2 + ||v||_2)^2}{\epsilon^3} ||u-v||_2\label{eq:lip of f in H2}
\end{align}
where constants $\bar{L}_1, \bar{L}_2, \bar{L}_3$ are independent of $\epsilon$, $u$, and $v$, and are defined in (\ref{eq:lipsch const def H2}). Also, for $u\in W$, we have
\begin{align}
||-\nabla PD^\epsilon(u)||_2 &\leq \dfrac{\bar{L}_4 ||u||_2 + \bar{L}_5 ||u||^2_2}{\epsilon^{5/2}} \label{eq:bound on f in H2},
\end{align}
where constants are independent of $\epsilon$ and $u$ and are defined in (\ref{eq:lipsch const def H2 2}).
\end{theorem}

In Theorem 6.1 of \cite{CMPer-Lipton}, the Lipschitz property of the peridynamic force is shown in $L^2(D;\R^d)$ and is given by
\begin{align}\label{eq:lipschitz const ltwo}
\Vert{-\nabla PD^\epsilon(u) - (-\nabla PD^\epsilon(v))}\Vert &\leq \dfrac{L_1}{\epsilon^2}\Vert{u - v}\Vert \,\,\,\,\, \forall
u ,v \in L^2(D;\R^d),
\end{align}
with $L_1$ given by \eqref{L1def}.

%

We state the theorem which shows the existence and uniqueness of solution in any given finite time interval $I_0 = (-T,T)$. 

\begin{theorem}\label{thm:existence over finite time domain} 
\textbf{Existence and uniqueness of solutions over finite time intervals}\\
For any initial condition $x_0\in X = W \times W$, time interval $I_0=(-T,T)$, and right hand side $b(t)$  continuous in time for $t\in I_0$ such that $b(t)$ satisfies $\sup_{t\in I_0} ||b(t)||_2<\infty$, there is a unique solution $y(t)\in C^1(I_0;X)$ of
\begin{equation*}
y(t)=x_0+\int_0^tF^\epsilon(y(\tau),\tau)\,d\tau,
\label{10}
\end{equation*}
or equivalently
\begin{equation*}
y'(t)=F^\epsilon(y(t),t),\hbox{with    $y(0)=x_0$},
\label{11}
\end{equation*}
where $y(t)$ and $y'(t)$ are Lipschitz continuous in time for $t\in I_0$.
\end{theorem}

It is found that the peridynamic evolutions have higher regularity in time for body forces that are differentiable in time.
We now state the higher temporal regularity for peridynamic evolutions.

\begin{theorem}\label{thm:higher regularity} 
\textbf{Higher regularity}\\
Suppose the initial data and righthand side $b(t)$ satisfy the hypothesis of Theorem \ref{thm:existence over finite time domain} and suppose further that  $\dot{b}(t)$ exists and is continuous in time for $t\in I_0$ and $\sup_{t\in I_0} ||\dot{b}(t)||_2 < \infty$. Then 
$u \in C^3(I_0;W)$ and 
\begin{align}\label{eq:per equation vsubtt est}
|| \partial^3_{ttt} u(t,x)||_2 &\leq  \dfrac{C(1 + \sup_{s \in I_0} ||u(s)||_2 + \sup_{s\in I_0} ||u(s)||^2_2)}{\epsilon^3} \sup_{s\in I_0} ||\partial_t u(s)||_2 + ||\dot{b}(t,x)||_2,
\end{align}
where $C$ is a positive constant independent of $u$.

\end{theorem}
The proofs of Theorems \ref{thm:lip in H2}, \ref{thm:existence over finite time domain}, and \ref{thm:higher regularity} are given in Section \ref{s:proofs}. We now discuss the finite element approximation of the peridynamic evolution.

\section{Finite element interpolation}\label{s:fe discr}
Let $V_h$ be given by linear continuous interpolations
over  tetrahedral or triangular elements $\mathcal{T}_h$ where $h$ denotes the size of finite element mesh. Here we  assume the elements are conforming and the finite element mesh is shape regular and $V_h\subset H^1_0(D;\R^d)$.

For a continuous function $u$ on $\bar{D}$, $\mathcal{I}_h(u)$ is the continuous piecewise linear interpolant on $\mathcal{T}_h$. It is given by 
\begin{align}
\mathcal{I}_h(u)\bigg\vert_{T} = \mathcal{I}_T( u) \qquad \forall T\in \mathcal{T}_h,
\end{align}
where $\mathcal{I}_T( u)$ is the local interpolant defined over finite element $T$ and is given by
\begin{align}
\mathcal{I}_T( u) = \sum_{i=1}^n u(x_i)\phi_i.
\end{align}
Here $n$ is the number of vertices in an element $T$, $x_i$ is the position of vertex $i$, and $\phi_i$ is the linear interpolant associated to vertex $i$. 

Application of Theorem 4.4.20 and remark 4.4.27 in \cite{MANa-Susanne} gives 
\begin{align}\label{eq:interpolation error}
&&|| u - \mathcal{I}_h(u) || \leq c h^2 || u ||_2, \hbox{         } \qquad \forall u \in W.
\end{align}

Let $r_h(u)$ denote the projection of $u\in W$ on $V_h$. For the $L^2$ norm it is defined as
\begin{align}
||u - r_h(u)|| &= \inf_{\tilde{u}\in V_h} ||u - \tilde{u}|| . \label{eq:proj u general}
\end{align}
and satisfies 
\begin{align}\label{eq:proj orthognl prop}
(r_h(u), \tilde{u}) = (u, \tilde{u}), \qquad \forall \tilde{u} \in V_h.
\end{align}

Since $\mathcal{I}_h(u) \in V_h$, and  \eqref{eq:interpolation error} we see that
\begin{align}
&&||u - r_h(u)|| \leq c h^2 || u ||_2,\hbox{     } \qquad \forall u \in W. \label{eq:proj error u general}
\end{align}

\subsection{Semi-discrete approximation}
Let $u_h(t) \in V_h$ be the approximation of $u(t)$ which satisfies the following 
\begin{align}
(\ddot{u}_h(t), \tilde{u} ) + a^\epsilon(u_h(t), \tilde{u}) &= ( b(t), \tilde{u} ), \qquad \forall \tilde{u} \in V_h, \label{eq:fe weak}.
\end{align}

We now show that the semi-discrete approximation is stable, i.e. the energy at time $t$ is bounded by the initial energy and work done by the body force. 

\begin{theorem}\label{thm:stab nonlin semi}
\textbf{Energy stability of the semi-discrete approximation}\\
The semi-discrete scheme is energetically stable and the energy $\mathcal{E}^\epsilon(u_h)(t)$, defined in (\ref{eq:def energy}), satisfies the following bound
\begin{align*}
\mathcal{E}^\epsilon(u_h)(t) &\leq \left[ \sqrt{\mathcal{E}^\epsilon(u_h)(0)} + \int_0^t ||b(\tau)|| d\tau \right]^2.
\end{align*}
\end{theorem}

\begin{proof}
Letting $\tilde{u} = \dot{u}_h(t)$ in (\ref{eq:fe weak}) and applying the identity (\ref{eq:energy relat}), we get
\begin{align*}
\dfrac{d}{dt} \mathcal{E}^\epsilon(u_h)(t) = (b(t), \dot{u}_h(t)) \leq ||b(t)|| \;||\dot{u}_h(t)||.
\end{align*}
We also have 
\begin{align*}
||\dot{u}_h(t)|| \leq 2 \sqrt{\frac{1}{2} ||\dot{u}_h||^2 + PD^\epsilon(u_h(t))} = 2 \sqrt{\mathcal{E}^\epsilon(u_h)(t)}
\end{align*}
where we used the fact that $PD^\epsilon(u)(t)$ is nonnegative and 
\begin{align*}
\dfrac{d}{dt} \mathcal{E}^\epsilon(u_h)(t) &\leq 2 \sqrt{\mathcal{E}^\epsilon(u_h)(t)}\; ||b(t)||.
\end{align*}

Fix $\delta > 0$ and let $A(t) = \mathcal{E}^\epsilon(u_h(t)) + \delta$. Then from the equation above we easily see that
\begin{align*}
\dfrac{d}{dt} A(t) &\leq 2 \sqrt{A(t)} \; ||b(t)|| \quad \Rightarrow \dfrac{1}{2} \dfrac{\frac{d}{dt} A(t)}{\sqrt{A(t)}} \leq ||b(t)||.
\end{align*}
Noting that $\frac{1}{\sqrt{a(t)}}\frac{d a(t)}{dt}= 2\frac{d}{dt} \sqrt{a(t)} $, integrating from $t=0$ to $\tau$ and relabeling $\tau$ as $t$, we get
\begin{align*}
\sqrt{A(t)} &\leq \sqrt{A(0)} + \int_0^t ||b(s)|| ds.
\end{align*}
Letting $\delta \to 0$ and taking the square of both side proves the claim.
\end{proof}

\section{Central difference time discretization}\label{s:central}
For illustration, we consider the central difference scheme and present the convergence rate for the central difference scheme for the fully nonlinear problem. We  point out that the extension of these results to the general Newmark scheme is straight forward. We conclude by considering a linearized peridynamic evolution and demonstrate CFL like conditions for stability of the fully discrete scheme.

Let $\Delta t$ be the time step. The exact solution at $t^k= k\Delta t$ (or time step $k$) is denoted as $(u^k, v^k)$, with $v^k=\partial u^k/\partial t$, and the projection onto $V_h$ at $t^k$ is given by $(r_h(u^k), r_h(v^k))$. The solution of the discrete problem at time step $k$ is denoted as $(u^k_h, v^k_h)$.  

We first describe the discrete approximation of the weak formulation of the peridynamic evolution. The initial data for displacement $u_0$ and velocity $v_0$ are approximated by their projections $r_h(u_0)$ and $r_h(v_0)$. Let $u^0_h = r_h(u_0)$ and $v^0_h = r_h(v_0)$. The discrete weak formulation of the peridynamic evolution is given as follows.  For $k\geq 1$, $(u^k_h, v^k_h)$ satisfies, for all $\tilde{u} \in V_h$,
\begin{align}
\left( \dfrac{u^{k+1}_h - u^k_h}{\Delta t}, \tilde{u} \right) &= (v^{k+1}_h, \tilde{u}), \notag \\
\left( \dfrac{v^{k+1}_h - v^k_h}{\Delta t}, \tilde{u} \right) &= (-\nabla PD^\epsilon(u^k_h), \tilde{u} ) + (b^k_h, \tilde{u}) \label{eq:forward},
\end{align}
where we denote projection of $b(t^k)$, $r_h(b(t^k))$, as $b^k_h$. Combining the two equations delivers the central difference equation for $u^k_h$. We have
\begin{align}\label{eq:central}
\left( \dfrac{u^{k+1}_h - 2 u^k_h + u^{k-1}_h}{\Delta t^2}, \tilde{u} \right) &= (-\nabla PD^\epsilon(u^k_h), \tilde{u} ) + (b^k_h, \tilde{u}), \qquad \forall \tilde{u} \in V_h.
\end{align}
For $k=0$, we have $\forall \tilde{u} \in V_h$
\begin{align}\label{eq:central zero}
\left( \dfrac{u^1_h - u^0_h}{\Delta t^2}, \tilde{u}\right) &= \dfrac{1}{2}(-\nabla PD^\epsilon(u^0_h), \tilde{u})+ \dfrac{1}{\Delta t} (v^0_h, \tilde{u}) + \dfrac{1}{2}(b^0_h, \tilde{u}).
\end{align}


Next, we describe the discrete approximation of the peridynamic evolution written in strong form. Let $u^k_h \in V_h$ be given by
\begin{align}
u_h^k\bigg\vert_{T} = \sum_{i=1}^n u_i^k\phi_i(x),\qquad x\in T, \qquad \forall T\in \mathcal{T}_h.
\end{align}
The discrete solution is represented as follows
\begin{align}\label{eq:central strong}
\dfrac{u^{k+1}_h - 2 u^k_h + u^{k-1}_h}{\Delta t^2}&= -\nabla PD^\epsilon_h(u^k_h) + b^k_h
\end{align}
when $k \geq 1$, and
\begin{align}\label{eq:central zero strong}
\dfrac{u^1_h - u^0_h}{\Delta t^2} &= -\dfrac{1}{2}\nabla PD^\epsilon_h(u^0_h) + \dfrac{1}{\Delta t} v^0_h + \dfrac{1}{2}b^0_h
\end{align}
when $k=0$. Here $-\nabla PD^\epsilon_h(u^k_h)$ and $b^k_h$ are defined as the continuous piecewise linear interpolation of $-\nabla PD^\epsilon(u^k_h)$ and $b^k$, i.e.,
\begin{align*}
-\nabla PD^\epsilon_h(u_h^k)\bigg\vert_{T} &= \sum_{i=1}^n (-\nabla PD^\epsilon(u_h^k)(x_i)) \phi_i(x), \qquad &x\in T, \qquad \forall T\in \mathcal{T}_h, \notag \\
b^k_h\bigg\vert_{T} &= \sum_{i=1}^n b^k(x_i) \phi_i(x),\qquad &x\in T, \qquad \forall T\in \mathcal{T}_h.
\end{align*}

\subsection{Implementation details}
For completeness we describe the implementation of the time stepping method using FEM interpolants  for both strong and weak form descriptions of the evolution.  Let $\bolds{N}$ be the shape tensor then $u^k_h, \tilde{u} \in V_h$ are given by
\begin{align}
u^k_h = \bolds{N} \bolds{U}^k, \qquad \tilde{u}  = \bolds{N} \tilde{\bolds{U}},
\end{align}
where $\bolds{U}^k$ and $\tilde{\bolds{U}}$ are $ Nd$ dimensional vectors, where $N$ is the number of nodal points in the mesh and $d$ is the dimension.  

\paragraph{Weak form}
From (\ref{eq:central}), for all $\tilde{\bolds{U}} \in \R^{Nd}$ with elements of $\tilde{\bolds{U}}$ zero on the boundary, then the following holds for $k\geq 1$
\begin{align}\label{eq:fe equation weak}
\left[\bolds{M}\frac{\bolds{U}^{k+1} - 2\bolds{U}^{k} + \bolds{U}^{k-1}}{\Delta t^2} \right] \cdot \tilde{\bolds{U}} = \bolds{F}^k \cdot \tilde{\bolds{U}} .
\end{align}
Here the mass matrix $\bolds{M}$ and force vector $\bolds{F}^k$ are given by
\begin{align}\label{eq:def M and F}
\bolds{M} &:= \int_D \bolds{N}^T \bolds{N} dx,\notag \\
\bolds{F}^k &:= \bolds{F}^k_{pd}+ \int_D \bolds{N}^T b(t^k,x) dx,
\end{align}
where $\bolds{F}^k_{pd}$ is defined by
\begin{align}\label{eq:def F pd}
\bolds{F}^k_{pd} &:= \int_D \bolds{N}^T (-\nabla PD^\epsilon(u^k_h)(x)) dx.
\end{align}
We remark that a similar equation holds for $k=0$.

At the time step $k$ we must invert ${\bolds{M}}$ to solve for ${\bolds{U}}^{k+1}$ using
\begin{align}\label{eq:weak form final solve}
{\bolds{U}}^{k+1} = \Delta t^2 {\bolds{M}}^{-1} {\bolds{F}}^k + 2 \hat{\bolds{U}}^k - {\bolds{U}}^{k-1}.
\end{align}
As is well known this inversion amounts to an increase of computational complexity associated with discrete approximation of the weak formulation of the evolution. On the other hand the quadrature error in the computation of the force vector $\bolds{F}^k_{pd}$ is reduced when using the weak form. 
\paragraph{Strong form}
In the strong formulation, the following holds for $k\geq 1$
\begin{align}\label{eq:fe equation strong}
\frac{\bolds{U}^{k+1} - 2\bolds{U}^{k} + \bolds{U}^{k-1}}{\Delta t^2} = \bolds{F}^k,
\end{align}
where the force vector is $\bolds{F}^k = (\bolds{F}^k_i)_{1\leq i\leq Nd} \in \R^{Nd}$ and $i^{\text{th}}$ element of $\bolds{F}^k$ is given by
\begin{align}
\bolds{F}^k_i &:= -\nabla PD_h^\epsilon(u^k_h)(x_i) + b(t^k,x_i).
\end{align}
The equation for $k=0$ can be written in a similar fashion. We solve for ${\bolds{U}}^{k+1}$ at the $k^{\text{th}}$ time step using
\begin{align}\label{eq:strong form final solve}
{\bolds{U}}^{k+1} = \Delta t^2 {\bolds{F}}^k + 2 {\bolds{U}}^k - {\bolds{U}}^{k-1}.
\end{align}

In contrast to the weak form approximation, the strong form approximation does not require the mass matrix  or its inverse. This amounts to a reduction in computational complexity when using the strong form approximation.  Additionally the computation of the force $\bolds{F}^k$ is much simpler for the strong form and only requires calculation of the peridynamic force at the mesh nodes. 


We now show that finite element solution of both the strong and weak formulations converge to the exact solution.

\subsection{Convergence of approximation}\label{approxcomvg}
The convergence rates are the same for both weak and strong form time stepping methods. We show how to establish  a uniform bound on the $L^2$ norm of the discretization error for the problem in weak form and prove that approximate solution converges to the exact solution at the rate $C_t\Delta t +C_s h^2/\epsilon^2$ for fixed $\epsilon > 0$. Identical arguments can be made for discretization error using the strong form approximation.  We first compare the exact solution with its projection in $V_h$ and then compare the projection with approximate solution. We further divide the calculation of error between projection and approximate solution in two parts, namely consistency analysis and error analysis.

Error $E^k$ is given by
\begin{align*}
E^k := ||u^k_h - u(t^k)|| + ||v^k_h - v(t^k)||.
\end{align*}
We split the error as follows
\begin{align*}
E^k &\leq \left( ||u^k - r_h(u^k)|| + ||v^k - r_h(v^k)|| \right) + \left( ||  u^k_h-r_h(u^k) || + ||v^k_h-r_h(v^k) || \right),
\end{align*}
where first term is error between exact solution and projections, and second term is error between projections and approximate solution. Let 
\begin{align}\label{eq:def ek}
e^k_h(u) := u^k_h-r_h(u^k) \quad \text{ and  } \quad e^k_h(v) :=v^k_h- r_h(v^k)
\end{align}
and
\begin{align}\label{eq:def total ek}
e^k := ||e^k_h(u)|| + ||e^k_h(v)||.
\end{align}

Using (\ref{eq:proj error u general}), we have
\begin{align}\label{eq:Ek ineq}
E^k &\leq C_p h^2 + e^k,
\end{align}
where
\begin{align}\label{eq:const Cp}
C_p :=  c\left[ \sup_t ||u(t)||_2 + \sup_t ||\dfrac{\partial u(t)}{\partial t}||_2 \right].
\end{align}

We have following main result
\begin{theorem}\label{thm:convergence}
\textbf{Convergence of the central difference approximation with respect to the $L^2$ norm.}\\
Let $(u,v)$ be the exact solution of the peridynamics equation in (\ref{eq:per equation}). Let $(u^k_h, v^k_h)$ are the FE approximate solution of (\ref{eq:central}) and \eqref{eq:central zero}. If $u,v \in C^2([0,T]; W)$, then the scheme is consistent and the error $E^k$ satisfies following bound
\begin{align}\label{eq:Ek bound}
&\sup_{k \leq T/\Delta t} E^k \notag \\
&= C_p h^2 + \exp[T(1+L_1/\epsilon^2)(\frac{1}{1-\Delta t})] \left[ e^0 + \left(\frac{T}{1-\Delta t}\right) \left(C_t \Delta t + C_s \dfrac{h^2}{\epsilon^2} \right)  \right]
\end{align}
where the constants $C_p$, $C_t$, and $C_s$ are given by (\ref{eq:const Cp}) and (\ref{eq:const Ct and Cs}). Here the constant $L_1/\epsilon^2$ is the Lipschitz constant of $-\nabla PD^\epsilon(u)$ in $L^2$, see \eqref{eq:lipschitz const ltwo} and (\ref{L1def}). 
If the error in initial data is zero then $E^k$ is of the order of $C_t\Delta t + C_s h^2/\epsilon^2$.
The same convergence rate holds for the discreet  evolution given in strong form by \eqref{eq:central strong} and \eqref{eq:central zero strong}.
\end{theorem} 
Theorem \ref{thm:higher regularity}  shows that  $u,v \in C^2([0,T]; W)$ for righthand side $b\in C^1([0,T]; W)$.
In section \ref{s:discussion} we discuss the behavior of the exponential constant appearing in Theorem \ref{thm:convergence}  for evolution times seen in fracture experiments.
Since we are approximating the solution of an ODE on a Banach space the proof of Theorem
\ref{thm:convergence} will follow from the Lipschitz continuity of the force $\nabla PD^\epsilon(u)$ with respect to the $L^2$ norm. The proof is given in the following two sections.

\subsubsection{Truncation error analysis and consistency}
We derive the equation for evolution of $e^k_h(u)$ as follows
\begin{align*}
&\left( \dfrac{u^{k+1}_h - u^k_h}{\Delta t} - \dfrac{r_h(u^{k+1}) - r_h(u^k)}{\Delta t} , \tilde{u}\right)  \notag \\
&= (v^{k+1}_h, \tilde{u}) - \left( \dfrac{r_h(u^{k+1}) - r_h(u^k)}{\Delta t}, \tilde{u} \right) \notag \\
&= (v^{k+1}_h , \tilde{u}) - (r_h(v^{k+1}) , \tilde{u}) + (r_h(v^{k+1}), \tilde{u}) - (v^{k+1}, \tilde{u}) \notag \\
&\quad + (v^{k+1}, \tilde{u}) - \left( \dfrac{\partial u^{k+1}}{\partial t}, \tilde{u}\right) \notag \\
&\quad + \left( \dfrac{\partial u^{k+1}}{\partial t}, \tilde{u}\right) - \left(\dfrac{u^{k+1} - u^k}{\Delta t}, \tilde{u} \right) \notag \\
&\quad + \left( \dfrac{u^{k+1} - u^k}{\Delta t}, \tilde{u} \right) - \left(\dfrac{r_h(u^{k+1}) - r_h(u^k)}{\Delta t}, \tilde{u} \right).
\end{align*}
Using property $(r_h(u), \tilde{u}) = (u,\tilde{u})$ for $\tilde u \in V_h$ and the fact that $\frac{\partial u(t^{k+1})}{\partial t} = v^{k+1}$ where $u$ is the exact solution, we get
\begin{align}
\left( \dfrac{e^{k+1}_h(u) - e^k_h(u)}{\Delta t}, \tilde{u} \right) = (e^{k+1}_h(v), \tilde{u}) + \left( \dfrac{\partial u^{k+1}}{\partial t}, \tilde{u}\right) - \left(\dfrac{u^{k+1} - u^k}{\Delta t}, \tilde{u} \right).
\end{align}

Let $(\tau^k_h(u), \tau^k_h(v))$ be the truncation error in the time discretization given by
\begin{align*}
\tau^k_h(u) &:= \dfrac{\partial u^{k+1}}{\partial t} - \dfrac{u^{k+1} - u^k}{\Delta t}, \\
\tau^k_h(v) &:= \dfrac{\partial v^k}{\partial t} - \dfrac{v^{k+1} - v^k}{\Delta t}. 
\end{align*}

With the above notation, we have
\begin{align}
( e^{k+1}_h(u), \tilde{u}) &= (e^{k}_h(u), \tilde{u}) + \Delta t (e^{k+1}_h(v), \tilde{u} ) + \Delta t (\tau^k_h(u),\tilde{u}) . \label{eq:ek u}
\end{align}

We now derive the equation for $e^k_h(v)$ as follows
\begin{align*}
&\left( \dfrac{v^{k+1}_h - v^k_h}{\Delta t} - \dfrac{r_h(v^{k+1}) - r_h(v^k)}{\Delta t}, \tilde{u}\right)  \notag \\
&= (-\nabla PD^\epsilon(u^k_h), \tilde{u} ) + (b^k_h, \tilde{u}) - \left( \dfrac{r_h(v^{k+1}) - r_h(v^k)}{\Delta t}, \tilde{u} \right) \notag \\
&= (-\nabla PD^\epsilon(u^k_h), \tilde{u} ) + (b^k, \tilde{u}) - \left( \dfrac{\partial v^k}{\partial t}, \tilde{u} \right)\notag \\
&\quad + \left( \dfrac{\partial v^k}{\partial t}, \tilde{u}\right) - \left( \dfrac{v^{k+1} - v^k}{\Delta t}, \tilde{u} \right) \notag \\
&\quad  + \left(\dfrac{v^{k+1} - v^k}{\Delta t}, \tilde{u} \right) - \left(\dfrac{r_h(v^{k+1}) - r_h(v^k)}{\Delta t}, \tilde{u} \right) \notag \\
&= \left( -\nabla PD^\epsilon(u^k_h)+ \nabla PD^\epsilon(u^k), \tilde{u} \right) + (b^k_h - b(t^k), \tilde{u}) \notag \\
&\quad + \left( \dfrac{\partial v^k}{\partial t}, \tilde{u}\right) - \left( \dfrac{v^{k+1} - v^k}{\Delta t}, \tilde{u} \right)  + \left(\dfrac{v^{k+1} - v^k}{\Delta t}, \tilde{u} \right) - \left(\dfrac{r_h(v^{k+1}) - r_h(v^k)}{\Delta t}, \tilde{u} \right) \notag \\
&= \left( -\nabla PD^\epsilon(u^k_h)+ \nabla PD^\epsilon(u^k), \tilde{u} \right)  + \left( \dfrac{\partial v^k}{\partial t}- \dfrac{v^{k+1} - v^k}{\Delta t}, \tilde{u} \right)
\end{align*}
where we used the property of $r_h(u)$ and the fact that 
\begin{align*}
(-\nabla PD^\epsilon(u^k), \tilde{u} ) + (b^k, \tilde{u}) - \left( \dfrac{\partial v^k}{\partial t}, \tilde{u} \right) = 0, \quad \forall \tilde{u} \in V_h.
\end{align*}

We further divide the error in the peridynamics force as follows
\begin{align*}
&\left( -\nabla PD^\epsilon(u^k_h)+ \nabla PD^\epsilon(u^k), \tilde{u} \right) \notag \\
&=\left( -\nabla PD^\epsilon(u^k_h)+ \nabla PD^\epsilon(r_h(u^k)), \tilde{u} \right) + \left( -\nabla PD^\epsilon(r_h(u^k))+ \nabla PD^\epsilon(u^k), \tilde{u} \right).
\end{align*}
We will see in next section that second term is related to the truncation error in the spatial discretization. Therefore, we define another truncation error term $\sigma^k_{per,h}(u)$ as follows
\begin{align}\label{eq:consistency error u peridynamics force}
\sigma^k_{per,h}(u) &:= -\nabla PD^\epsilon(r_h(u^k)) + \nabla PD^\epsilon(u^k).
\end{align}

After substituting the notations related to truncation errors, we get
\begin{align}
(e^{k+1}_h(v), \tilde{u}) &= (e^k_h(v), \tilde{u}) + \Delta t (-\nabla PD^\epsilon(u^k_h) + \nabla PD^\epsilon(r_h(u^k)), \tilde{u})  \notag \\
&\; + \Delta t (\tau^k_h(v), \tilde{u})+ \Delta t (\sigma^k_{per,h}(u), \tilde{u}). \label{eq:ek v}
\end{align}

When $u,v$ are $C^2$ in time, we easily see
\begin{align*}
||\tau^k_h(u)|| &\leq \Delta t \sup_{t} ||\frac{\partial^2 u}{\partial t^2}|| \qquad \text{and} \qquad ||\tau^k_h(v)|| \leq \Delta t \sup_{t} ||\frac{\partial^2 v}{\partial t^2}||. 
\end{align*}
Here $u$ and $v$ are $C^2$ in time for differentiable in time body forces as stated in Theorem
\ref{thm:higher regularity} and Theorem \ref{thm:soln of v equation}.

To estimate $\sigma^k_{per,h}(u)$, we note the Lipschitz property of the peridynamics force in $L^2$ norm, see (\ref{eq:lipschitz const ltwo}). This leads us to 
\begin{align}
||\sigma^k_{per,h}(u)|| &\leq \dfrac{L_1}{\epsilon^2} ||u^k - r_h(u^k)|| \leq \dfrac{L _1c}{\epsilon^2} h^2 \sup_{t} ||u(t)||_2. \label{eq:consistency error u in spatial peridynamics force}
\end{align}

We now state the consistency of this approach.

\begin{lemma}\label{lemma: constency}
\textbf{Consistency}\\
Let $\tau$ be given by
\begin{align}
\tau &:= \sup_{k} \left( ||\tau^k_h(u)|| + ||\tau^k_h(v)|| + ||\sigma^k_{per,h}(u)|| \right), \label{tau} 
\end{align}
then the approach is consistent in that
\begin{align}
\tau&\leq C_t \Delta t + C_s \frac{h^2}{\epsilon^2} . \label{eq:total consistency error}
\end{align}
where
\begin{align}
C_t &:= \sup_{t}||\frac{\partial^2 u}{\partial t^2}|| + \sup_{t} ||\frac{\partial^2 v}{\partial t^2}|| \quad \text{and} \quad C_s := L_1 c \sup_{t} ||u(t)||_2. \label{eq:const Ct and Cs}
\end{align}
\end{lemma}

\subsubsection{Stability analysis}In equation for $e^k_h(u)$, see (\ref{eq:ek u}), we take $\tilde{u} = e^{k+1}_h(u)$. Note that $e^{k+1}_h(u) = u^k_h - r_h(u^k) \in V_h$. We have
\begin{align*}
||e^{k+1}_h(u)||^2 &= (e^{k}_h(u), e^{k+1}_h(u)) + \Delta t (e^{k+1}_h(v), e^{k+1}_h(u)) + \Delta t (\tau^k_h(u), e^{k+1}_h(u)),
\end{align*}
and we get
\begin{align*}
||e^{k+1}_h(u)||^2 &\leq ||e^{k}_h(u)|| \;||e^{k+1}_h(u)|| + \Delta t ||e^{k+1}_h(v)||\; ||e^{k+1}_h(u)|| + \Delta t  ||\tau^{k}_h(u)|| \; ||e^{k+1}_h(u)||.
\end{align*}
Canceling $||e^{k+1}_h(u)||$ from both sides gives
\begin{align}\label{eq:Ltwo ek u}
||e^{k+1}_h(u)|| &\leq ||e^{k}_h(u)|| + \Delta t ||e^{k+1}_h(v)|| + \Delta t ||\tau^{k}_h(u)|| .
\end{align}

Similarly, if we choose $\tilde{u} = e^{k+1}_h(v)$ in (\ref{eq:ek v}), and use the steps similar to above, we get
\begin{align}\label{eq:Ltwo ek v}
||e^{k+1}_h(v)|| &\leq ||e^{k}_h(v)|| + \Delta t ||-\nabla PD^\epsilon(u^k_h) + \nabla PD^\epsilon(r_h(u^k))|| \notag \\
& \quad + \Delta t \left( ||\tau^{k}_h(v)|| +  ||\sigma^{k}_{per,h}(u)|| \right).
\end{align}

Using the Lipschitz property of the peridynamics force in $L^2$, we have
\begin{align}\label{eq:error in approx u peridynamics force}
||-\nabla PD^\epsilon(u^k_h) + \nabla PD^\epsilon(r_h(u^k))|| &\leq \dfrac{L_1}{\epsilon^2} ||u^k_h  - r_h(u^k)|| = \dfrac{L_1}{\epsilon^2} ||e^k_h(u)||.
\end{align}

After adding (\ref{eq:Ltwo ek u}) and (\ref{eq:Ltwo ek v}), and substituting (\ref{eq:error in approx u peridynamics force}), we get
\begin{align*}
||e^{k+1}_h(u)|| + ||e^{k+1}_h(v)|| &\leq ||e^k_h(u)|| + ||e^k_h(v)|| + \Delta t ||e^{k+1}_h(v)|| + \dfrac{L_1}{\epsilon^2} \Delta t ||e^k_h(u)|| + \Delta t \tau
\end{align*}
where $\tau$ is defined in (\ref{eq:total consistency error}). 

Let $e^k := ||e^k_h(u)|| + ||e^k_h(v)||$. Assuming $L_1/\epsilon^2 \geq 1$, we get
\begin{align*}
&e^{k+1} \leq e^k + \Delta t e^{k+1} + \Delta t \dfrac{L_1}{\epsilon^2} e^k + \Delta t \tau \\
\Rightarrow & e^{k+1} \leq \dfrac{1 + \Delta t L_1/\epsilon^2}{1-\Delta t} e^k + \dfrac{\Delta t}{1 - \Delta t} \tau.
\end{align*}

Substituting for $e^k$ recursively in the equation above, we get
\begin{align*}
e^{k+1} &\leq \left(\dfrac{1 + \Delta t L_1/\epsilon^2}{1-\Delta t} \right)^{k+1} e^0 + \dfrac{\Delta t}{1 - \Delta t} \tau \sum_{j=0}^k \left(\dfrac{1 + \Delta t L_1/\epsilon^2}{1-\Delta t} \right)^{k-j}.
\end{align*}
Noting $(1 +a \Delta t )^k \leq \exp[k a\Delta t ] \leq \exp[Ta]$ for $a>0$ and
\begin{align*}
\dfrac{1 + \Delta t L_1/\epsilon^2}{1-\Delta t} &= 1 + \frac{(1+L_1/\epsilon^2)}{1-\Delta t} \Delta t
\end{align*}
we get
\begin{align*}
\left(\dfrac{1 + \Delta t L_1/\epsilon^2}{1-\Delta t} \right)^k &\leq \exp[\frac{T(1+L_1/\epsilon^2)}{1-\Delta t}].
\end{align*}

Substituting above estimates, we can easily show that
\begin{align*}
e^{k+1} &\leq \exp[\frac{T(1+L_1/\epsilon^2)}{1-\Delta t}] \left[ e^0 + \dfrac{\Delta t}{1 - \Delta t} \tau \sum_{j=0}^k 1 \right]\\
&\leq \exp[\frac{T(1+L_1/\epsilon^2)}{1-\Delta t}] \left[ e^0 + \dfrac{k\Delta t}{1 - \Delta t} \tau\right].
\end{align*}

Finally, we substitute above into (\ref{eq:Ek ineq}) to conclude

\begin{lemma}\label{lemma: stability}
\textbf{Stability}\\
\begin{align}\label{LaxRichtmyer}
E^k &\leq C_p h^2 + \exp[\frac{T(1+L_1/\epsilon^2)}{1-\Delta t}] \left[ e^0 + \dfrac{k\Delta t}{1 - \Delta t} \tau\right].
\end{align}
After taking sup over $k\leq T/\Delta t$ and substituting the bound on $\tau$ from Lemma \ref{lemma: constency}, we get the desired result and proof of Theorem \ref{thm:convergence} is complete.
\end{lemma}

%
%
%
%

We now consider a stronger notion of stability for the linearized peridynamics model.

\subsection{Linearized peridynamics and energy stability}
In this section, we linearize the peridynamics model and obtain a CFL like stability condition. For problems where strains are small, the stability condition for the linearized model is expected to apply to the nonlinear model. The slope of peridynamics potential $f$ is constant for sufficiently small strain and therefore for small strain the nonlinear model behaves like a linear model. When displacement field is smooth, the difference between the linearized peridynamics force and the nonlinear peridynamics force is of the order of $\epsilon$. See [Proposition 4, \citet{CMPer-JhaLipton2}]. 

For strain far below the critical strain, ie., $|S(u)|<<S_c$ we expand  the integrand of (\ref{eq:per force}) in a Taylor series about zero to obtain the linearized peridymamic force given by
\begin{align}\label{eq:per force linear}
-\nabla PD^{\epsilon}_l(u)(x) = \dfrac{4}{\epsilon^{d+1} \omega_d} \int_{H_{\epsilon}(x)} \omega(x) \omega(y) J^\epsilon(|y - x|) f'(0) S(u) e_{y - x} dy.
\end{align}
The corresponding bilinear form is denoted as $a^\epsilon_l$ and is given by
\begin{align}
a^\epsilon_l(u, v) &= \dfrac{2}{\epsilon^{d+1} \omega_d} \int_D \int_{H_\epsilon(x)} \omega(x) \omega(y) J^\epsilon(|y - x|) f'(0) |y - x| S(u) S(v) dy dx. \label{eq:operator a linear}
\end{align}
We have
\begin{align*}
(-\nabla PD^\epsilon_l(u), v) = - a^\epsilon_l(u,v).
\end{align*}
We now discuss the stability of the FEM approximation to the linearized problem. 
We replace  $-\nabla PD^\epsilon$ by its linearization denoted by $-\nabla PD^\epsilon_l$ in (\ref{eq:central}) and (\ref{eq:central zero}).  The corresponding approximate solution in $V_h$ is denoted by $u^k_{l,h}$  where
\begin{align}\label{eq:central lin}
\left( \dfrac{u^{k+1}_{l,h} - 2 u^k_{l,h} + u^{k-1}_{l,h}}{\Delta t^2}, \tilde{u} \right) &= (-\nabla PD^\epsilon_l(u^k_{l,h}), \tilde{u} ) + (b^k_{h}, \tilde{u}), \qquad \forall \tilde{u} \in V_h
\end{align}
and
\begin{align}\label{eq:central zero lin}
\left( \dfrac{u^1_{l,h} - u^0_{l,h}}{\Delta t^2}, \tilde{u}\right) &= \dfrac{1}{2}(-\nabla PD^\epsilon(u^0_{l,h}), \tilde{u})+ \dfrac{1}{\Delta t} (v^0_{l,h}, \tilde{u}) + \dfrac{1}{2}(b^0_h, \tilde{u}), \qquad \forall \tilde{u} \in V_h.
\end{align}

We will adopt following notations 
\begin{align}
&\overline{u}^{k+1}_h := \frac{u^{k+1}_h + u^k_h}{2}, \: \overline{u}^{k}_h := \frac{u^k_h + u^{k-1}_h}{2}, \notag \\
&\bar{\partial}_t u^k_h := \frac{u^{k+1}_h - u^{k-1}_h}{2 \Delta t}, \: \bar{\partial}_t^+ u^k_h :=  \frac{u^{k+1}_h - u^{k}_h}{\Delta t}, \: \bar{\partial}_t^- u^k_h :=  \frac{u^{k}_h - u^{k-1}_h}{\Delta t}.
\end{align}
With above notations, we have
\begin{align*}
\bar{\partial}_t u^k_h &= \frac{\bar{\partial}^+_t u^k_h + \bar{\partial}^{-}_h u^k_h}{ 2} = \frac{\overline{u}^{k+1}_h - \overline{u}^{k}_h}{\Delta t}.
\end{align*}
We also define
\begin{align*}
\bar{\partial}_{tt} u^k_h &:= \dfrac{u^{k+1}_h - 2u^k_h + u^{k-1}_h}{\Delta t^2} = \dfrac{\bar{\partial}^+_t u^k_h - \bar{\partial}^-_t u^k_h}{\Delta t}.
\end{align*}
We introduce the discrete energy associated with $u^k_{l,h}$ at time step $k$ as defined by
\begin{align*}
\mathcal{E}(u^k_{l,h}) &:= \frac{1}{2} \left[  ||\bar{\partial}^+_t u^k_{l,h}||^2 - \frac{\Delta t^2}{4} a^\epsilon_l(\bar{\partial}^+_t u^k_{l,h}, \bar{\partial}^+_t u^k_{l,h}) + a^\epsilon_l(\overline{u}^{k+1}_{l,h} , \overline{u}^{k+1}_{l,h}) \right]
\end{align*}

Following [Theorem 4.1, \citet{CMPer-Karaa}], we have

\begin{theorem}\label{thm:cfl condition l}
\textbf{Energy Stability of the Central difference approximation of linearized peridynamics}

Let $u^k_{l,h}$ be the approximate solution of (\ref{eq:central lin}) and (\ref{eq:central zero lin}) with respect to linearized peridynamics. In the absence of body force $b(t) = 0$ for all $t$, if $\Delta t$ satisfies the CFL like condition
\begin{align}
\frac{\Delta t^2}{4} \sup_{u\in V_h \setminus \{0\}} \dfrac{a^\epsilon_l(u,u)}{(u,u)} \leq 1,
\end{align}
The discrete energy
is positive and we have the stability 
\begin{align}\label{energy}
\mathcal{E}(u^k_{l,h}) = \mathcal{E}(u^{k-1}_{l,h}).
\end{align}
\end{theorem}

\begin{proof}
Set $b(t) = 0$. Noting that $a^\epsilon_l$ is bilinear, after adding and subtracting term $(\Delta t^2/4)a^\epsilon_l(\bar{\partial}_{tt} u^k_{l,h}, \tilde{u})$ to (\ref{eq:central lin}), and noting following
\begin{align*}
u^k_{l,h} + \dfrac{\Delta t^2}{4} \bar{\partial}_{tt} u^k_{l,h} = \dfrac{\overline{u}^{k+1}_{l,h}}{2} + \dfrac{\overline{u}^{k}_{l,h}}{2}
\end{align*}
we get
\begin{align*}
(\bar{\partial}_{tt} u^k_{l,h}, \tilde{u}) -\dfrac{\Delta t^2}{4}a^\epsilon_l(\bar{\partial}_{tt} u^k_{l,h}, \tilde{u}) + \dfrac{1}{2} a^\epsilon_l(\overline{u}^{k+1}_{l,h} + \overline{u}^{k}_{l,h}, \tilde{u}) = 0.
\end{align*}
We let $\tilde{u} = \bar{\partial}_t u^k_{l,h}$, to write
\begin{align*}
(\bar{\partial}_{tt} u^k_{l,h}, \bar{\partial}_t u^k_{l,h}) -\dfrac{\Delta t^2}{4}a^\epsilon_l(\bar{\partial}_{tt} u^k_{l,h}, \bar{\partial}_t u^k_{l,h}) + \dfrac{1}{2} a^\epsilon_l(\overline{u}^{k+1}_{l,h}+\overline{u}^{k}_{l,h}, \bar{\partial}_t u^k_{l,h}) = 0.
\end{align*}
It is easily shown that
\begin{align*}
(\bar{\partial}_{tt} u^k_{l,h}, \bar{\partial}_t u^k_{l,h}) &= \left( \dfrac{\bar{\partial}_t^+ u^k_{l,h} - \bar{\partial}_t^- u^k_{l,h}}{\Delta t}, \dfrac{\bar{\partial}_t^+ u^k_{l,h} + \bar{\partial}_t^- u^k_{l,h}}{2} \right) = \dfrac{1}{2 \Delta t} \left[ ||\bar{\partial}_t^+ u^k_{l,h}||^2 - ||\bar{\partial}_t^- u^k_{l,h}||^2 \right]
\end{align*}
and
\begin{align*}
a^\epsilon_l(\bar{\partial}_{tt} u^k_{l,h}, \bar{\partial}_t u^k_{l,h}) &= \dfrac{1}{2\Delta t} \left[a^\epsilon_l(\bar{\partial}_t^+ u^k_{l,h}, \bar{\partial}_t^+ u^k_{l,h}) - a^\epsilon_l(\bar{\partial}_t^- u^k_{l,h}, \bar{\partial}_t^- u^k_{l,h}) \right].
\end{align*}
Noting that $\bar{\partial}_t u^k_{l,h} = (\overline{u}^{k+1}_{l,h} - \overline{u}^{k}_{l,h})/\Delta t$, we get
\begin{align*}
&\dfrac{1}{2\Delta t}a^\epsilon_l(\overline{u}^{k+1}_{l,h}+\overline{u}^{k}_{l,h}, \overline{u}^{k+1}_{l,h} - \overline{u}^{k}_{l,h}) \notag \\
&= \dfrac{1}{2 \Delta t} \left[a^\epsilon_l(\overline{u}^{k+1}_{l,h}, \overline{u}^{k+1}_{l,h}) - a^\epsilon_l(\overline{u}^{k}_{l,h}, \overline{u}^{k}_{l,h}) \right].
\end{align*}
After combining the above equations, we get
\begin{align}
\dfrac{1}{\Delta t} & \left[ \left(\dfrac{1}{2} ||\bar{\partial}_t^+ u^k_{l,h}||^2 - \dfrac{\Delta t^2}{8} a^\epsilon_l(\bar{\partial}_t^+ u^k_{l,h}, \bar{\partial}_t^+ u^k_{l,h}) +  \dfrac{1}{2} a^\epsilon_l(\overline{u}^{k+1}_{l,h}, \overline{u}^{k+1}_{l,h}) \right) \right. \notag \\
&\: - \left. \left(\dfrac{1}{2} ||\bar{\partial}_t^- u^k_{l,h}||^2 - \dfrac{\Delta t^2}{8} a^\epsilon_l(\bar{\partial}_t^- u^k_{l,h}, \bar{\partial}_t^- u^k_{l,h}) +  \dfrac{1}{2} a^\epsilon_l(\overline{u}^{k}_{l,h}, \overline{u}^{k}_{l,h}) \right) \right] = 0.\label{eq:stab l}
\end{align}
We recognize the first term in bracket as $\mathcal{E}(u^k_{l,h})$. We next prove that the second term is $\mathcal{E}(u^{k-1}_{l,h})$. We substitute $k = k-1$ in the definition of $\mathcal{E}(u^k_{l,h})$, to get
\begin{align*}
\mathcal{E}(u^{k-1}_{l,h}) &= \frac{1}{2} \left[  ||\bar{\partial}^+_t u^{k-1}_{l,h}||^2 - \frac{\Delta t^2}{4} a^\epsilon_l(\bar{\partial}^+_t u^{k-1}_{l,h}, \bar{\partial}^+_t u^{k-1}_{l,h}) + a^\epsilon_l(\overline{u}^{k}_{l,h} , \overline{u}^{k}_{l,h}) \right].
\end{align*}
We clearly have $\bar{\partial}^+_t u^{k-1}_{l,h} = \frac{u^{k-1+1}_{l,h} - u^{k-1}_{l,h}}{\Delta t} = \bar{\partial}^{-}_t u^k_{l,h}$ and this implies that the second term in (\ref{eq:stab l}) is $\mathcal{E}(u^{k-1}_{l,h})$.
It now follows from (\ref{eq:stab l}), that $\mathcal{E}(u^k_{l,h} )= \mathcal{E}(u^{k-1}_{l,h})$. 

We now establish the positivity of the discrete energy $\mathcal{E}(u^k_{l,h})$. To do this we derive a condition on the time step that insures the sum of the first two terms is positive and the positivity of $\mathcal{E}(u^k_{l,h})$ will follow. Let $v = \bar{\partial}^+_t u^k_{l,h} \in V_h$, then we require
\begin{align}\label{eq:stab l 1}
||v||^2 - \dfrac{\Delta t^2}{4}a^\epsilon_l(v,v) \geq 0 \quad \Rightarrow \dfrac{\Delta t^2}{4} \dfrac{a^\epsilon_l(v,v)}{||v||^2} \leq 1
\end{align}
Clearly if $\Delta t$ satisfies 
\begin{align}\label{eq:stab linear}
\dfrac{\Delta t^2}{4} \sup_{ v \in V_h \setminus \{0\}} \dfrac{a^\epsilon_l(v,v)}{||v||^2} \leq 1
\end{align}
then (\ref{eq:stab l 1}) is also satisfied and the discrete energy is positive.
Iteration gives  $\mathcal{E}(u^k_{l,h} )= \mathcal{E}(u^{0}_{l,h})$ and the theorem is proved.

\end{proof}

\section{Proof of claims}\label{s:proofs}
In this section we establish the Lipschitz continuity in the space $W = H^2(D;\R^d) \cap H^1_0(D;\R^d)$ and the existence of a differentiable in time solution to the peridynamic evolution belonging to 
 $W$. We outline the proof of Lipschitz continuity of $Q(v;u)$, see (\ref{eq:per equation v}), required to show the higher regularity of solutions in time.
 
\subsection{Lipschitz continuity in $ H^2\cap H^1_0$}
We now prove the Lipschitz continuity given by Theorem \ref{thm:lip in H2}. 

To simplify the presentation, we write the peridynamics force $-\nabla PD^\epsilon(u)$ as $P(u)$. 
We need to analyze $||P(u) - P(v)||_2$.

We first introduce the following convenient notations
\begin{align}
&s_\xi = \epsilon |\xi|, \: e_\xi = \dfrac{\xi}{|\xi|}, \: \bar{J}_\alpha = \dfrac{1}{\omega_d } \int_{H_1(0)} J(|\xi|) \dfrac{1}{|\xi|^\alpha} d\xi, \label{n1}\\
&S_\xi(u) = \dfrac{u(x+\epsilon \xi) - u(x)}{s_\xi} \cdot e_\xi,  \label{n2}\\
&S_\xi(\nabla u) = \nabla S_\xi(u) = \dfrac{\nabla u^T(x+\epsilon \xi) - \nabla u^T(x)}{s_\xi} e_\xi ,\label{n3}\\
&S_\xi(\nabla^2 u) = \nabla S_\xi(\nabla u) = \nabla \left[ \dfrac{\nabla u^T(x+\epsilon \xi) - \nabla u^T(x)}{s_\xi} e_\xi \right].\label{n4}
\end{align}
In indicial notation, we have
\begin{align}
S_\xi(\nabla u)_i &= \dfrac{ u_{k,i}(x+\epsilon \xi) - u_{k,i}(x)}{s_\xi} (e_\xi)_k, \notag \\
S_\xi(\nabla^2 u)_{ij} &=  \left[ \dfrac{ u_{k,i}(x+\epsilon \xi) - u_{k,i}(x)}{s_\xi} (e_\xi)_k \right]_{,j} = \dfrac{ u_{k,ij}(x+\epsilon \xi) - u_{k,ij}(x)}{s_\xi} (e_\xi)_k
\end{align}
and 
\begin{align}
[e_\xi \otimes S_\xi(\nabla^2 u) ]_{ijk} = (e_\xi)_i S_\xi(\nabla^2 u)_{jk},
\end{align}
where $u_{i,j} = (\nabla u)_{ij}$, $u_{k,ij} = (\nabla^2 u)_{kij}$, and $(e_\xi)_k = \xi_k/|\xi|$.  


We now examine the Lipschitz properties of the peridynamic force. Let $F_1(r) := f(r^2)$ where $f$ is described in the Section \ref{s:intro model}. We have $F_1'(r) = f'(r^2) 2r$. Thus, $2S f'(\epsilon |\xi| S^2)  = F_1'(\sqrt{\epsilon |\xi|}S)/\sqrt{\epsilon|\xi|}$. We define following constants related to nonlinear potential
\begin{align}
C_1 := \sup_r |F'_1(r)| , \: C_2 := \sup_r |F''_1(r)| , \: C_3 := \sup_r |F'''_1(r)| , \: C_4 := \sup_r |F''''_1(r)| .\label{Constants}
\end{align}
The potential function $f$ as chosen here satisfies $C_1, C_2, C_3, C_4 < \infty$. Let
\begin{align}
\bar{\omega}_\xi(x) = \omega(x) \omega(x+ \epsilon \xi),
\end{align}
and we choose $\omega$ such that
\begin{align}\label{eq:const C omega}
|\nabla \bar{\omega}_\xi| \leq C_{\omega_1} < \infty \quad \text{and} \quad |\nabla^2 \bar{\omega}_\xi| \leq C_{\omega_2} < \infty.
\end{align}

With notations described so far, we write peridynamics force $P(u)$ as
\begin{align}
P(u)(x) &= \dfrac{2}{\epsilon \omega_d} \int_{H_1(0)} \bar{\omega}_\xi(x) J(|\xi|)\dfrac{F_1'(\sqrt{s_\xi} S_\xi(u))}{\sqrt{s_\xi}} e_\xi d\xi. \label{eq:p u} 
\end{align}
The gradient of $P(u)(x)$ is given by
\begin{align}
\nabla P(u)(x) &= \dfrac{2}{\epsilon \omega_d} \int_{H_1(0)} \bar{\omega}_\xi(x) J(|\xi|)F_1''(\sqrt{s_\xi} S_\xi(u)) e_\xi \otimes  \nabla S_\xi(u) d\xi \notag \\
&\: + \dfrac{2}{\epsilon \omega_d} \int_{H_1(0)} J(|\xi|)\dfrac{F_1'(\sqrt{s_\xi} S_\xi(u))}{\sqrt{s_\xi}} e_\xi \otimes \nabla \bar{\omega}_\xi(x) d\xi \notag \\
&= g_1(u)(x) + g_2(u)(x), \label{eq:nabla p u}
\end{align}
where we denote first and second term as $g_1(u)(x)$ and $g_2(u)(x)$ respectively. We also have
\begin{align}
\nabla^2 P(u)(x) &= \dfrac{2}{\epsilon \omega_d} \int_{H_1(0)} \bar{\omega}_\xi(x) J(|\xi|)  F_1''(\sqrt{s_\xi} S_\xi(u)) e_\xi \otimes S_\xi(\nabla^2 u) d\xi \notag \\
&\: + \dfrac{2}{\epsilon \omega_d} \int_{H_1(0)} \bar{\omega}_\xi(x) J(|\xi|) \sqrt{s_\xi} F_1'''(\sqrt{s_\xi} S_\xi(u)) e_\xi \otimes S_\xi(\nabla u) \otimes S_\xi(\nabla u) d\xi \notag \\
&\: + \dfrac{2}{\epsilon \omega_d} \int_{H_1(0)} J(|\xi|) F_1''(\sqrt{s_\xi} S_\xi(u)) e_\xi \otimes S_\xi(\nabla u) \otimes \nabla \bar{\omega}_\xi(x)  d\xi \notag \\
&\: + \dfrac{2}{\epsilon \omega_d} \int_{H_1(0)} J(|\xi|)\dfrac{F_1'(\sqrt{s_\xi} S_\xi(u))}{\sqrt{s_\xi}} e_\xi \otimes \nabla^2 \bar{\omega}_\xi(x) d\xi \notag \\
&\: + \dfrac{2}{\epsilon \omega_d} \int_{H_1(0)} J(|\xi|) F_1''(\sqrt{s_\xi} S_\xi(u)) e_\xi \otimes \nabla \bar{\omega}_\xi(x) \otimes S_\xi(\nabla u) d\xi \notag \\
&= h_1(u)(x) + h_2(u)(x) + h_3(u)(x) + h_4(u)(x) + h_5(u)(x) \label{eq:nabla 2 p u}
\end{align}
where we denote first, second, third, fourth, and fifth terms as $h_1, h_2, h_2, h_4, h_5$ respectively. 

\paragraph{Estimating $||P(u) - P(v)||$}From (\ref{eq:p u}), we have
\begin{align}
|P(u)(x) - P(v)(x)| &\leq \dfrac{2}{\epsilon \omega_d} \int_{H_1(0)} J(|\xi|) \dfrac{1}{\sqrt{s_\xi}} |F'_1(\sqrt{s_\xi}S_\xi(u)) - F'_1(\sqrt{s_\xi}S_\xi(v))| d\xi \notag \\
&\leq \dfrac{2}{\epsilon \omega_d} \left(\sup_r |F_1'(r)| \right) \int_{H_1(0)} J(|\xi|) \dfrac{1}{\sqrt{s_\xi}} |\sqrt{s_\xi}S_\xi(u) - \sqrt{s_\xi}S_\xi(v)| d\xi \notag \\
&= \dfrac{2 C_2}{\epsilon \omega_d} \int_{H_1(0)} J(|\xi|) |S_\xi(u) - S_\xi(v)| d\xi, \label{eq:bd 1}
\end{align}
where we used the fact that $|\bar{\omega}_\xi(x) | \leq 1$ and $|F_1'(r_1) - F_1'(r_2)| \leq C_2 |r_1 - r_2|$. 

From (\ref{eq:bd 1}), we have
\begin{align*}
&||P(u) - P(v)||^2 \notag \\
&\leq \int_D \left( \dfrac{2 C_2}{\epsilon \omega_d} \right)^2 \int_{H_1(0)} \int_{H_1(0)} \dfrac{J(|\xi|)}{|\xi|} \dfrac{J(|\eta|)}{|\eta|} |\xi||S_\xi(u) - S_\xi(v)| |\eta||S_\eta(u) - S_\eta(v)| d\xi d\eta dx .
\end{align*}
Using the identities $|a||b| \leq |a|^2/2 + |b|^2/2$ and $(a + b)^2 \leq 2 a^2 + 2b^2$ we get
\begin{align}
&||P(u) - P(v)||^2  \notag \\
&\leq \int_D \left( \dfrac{2 C_2}{\epsilon \omega_d} \right)^2 \int_{H_1(0)} \int_{H_1(0)} \dfrac{J(|\xi|)}{|\xi|} \dfrac{J(|\eta|)}{|\eta|} \dfrac{|\xi|^2|S_\xi(u) - S_\xi(v)|^2 + |\eta|^2|S_\eta(u) - S_\eta(v)|^2}{2} d\xi d\eta dx \notag \\
&= 2 \int_D \left( \dfrac{2 C_2}{\epsilon \omega_d} \right)^2 \int_{H_1(0)} \int_{H_1(0)} \dfrac{J(|\xi|)}{|\xi|} \dfrac{J(|\eta|)}{|\eta|} \dfrac{|\xi|^2|S_\xi(u) - S_\xi(v)|^2}{2} d\xi d\eta dx \notag \\
&= \int_D \left( \dfrac{2 C_2}{\epsilon \omega_d} \right)^2 \omega_d \bar{J}_1 \int_{H_1(0)} \dfrac{J(|\xi|)}{|\xi|} |\xi|^2\dfrac{2|u(x+\epsilon \xi) - v(x+\epsilon \xi)|^2 + 2|u(x) - v(x)|^2}{\epsilon^2 |\xi|^2} d\xi  dx \notag \\
&= \left( \dfrac{2 C_2}{\epsilon \omega_d} \right)^2 \omega_d \bar{J}_1 \int_{H_1(0)} \dfrac{J(|\xi|)}{|\xi|} \dfrac{1}{\epsilon^2} \left[ 2\int_D \left( |u(x+\epsilon \xi) - v(x+\epsilon \xi)|^2 + |u(x) - v(x)|^2\right) dx \right] d\xi \notag \\
&\leq \left( \dfrac{2 C_2}{\epsilon \omega_d} \right)^2 \omega_d \bar{J}_1 \int_{H_1(0)} \dfrac{J(|\xi|)}{|\xi|} \dfrac{1}{\epsilon^2} \left[ 4 ||u-v||^2 \right] d\xi, \label{eq:bd 2}
\end{align}
where we used the symmetry with respect to $\xi$ and $\eta$ in second equation. This gives
\begin{align}
||P(u) - P(v)|| &\leq \dfrac{L_1}{\epsilon^2} ||u-v|| \leq  \dfrac{L_1}{\epsilon^2} ||u-v||_2, \label{eq:lipsch 2}
\end{align}
where 
\begin{align}
L_1 := 4 C_2 \bar{J}_1\label{L1def} .
\end{align}

\paragraph{Estimating $||\nabla P(u) - \nabla P(v)||$} From (\ref{eq:nabla p u}), we have
\begin{align*}
||\nabla P(u) - \nabla P(v)|| &\leq ||g_1(u) - g_1(v)|| + ||g_2(u) - g_2(v)||.
\end{align*}
Using $|\bar{\omega}_\xi(x)| \leq 1$, we get
\begin{align}
& |g_1(u)(x) - g_1(v)(x)| \notag \\
&\leq \dfrac{2}{\epsilon \omega_d} \int_{H_1(0)} J(|\xi|) |F_1''(\sqrt{s_\xi} S_\xi(u)) \nabla S_\xi(u) - F_1''(\sqrt{s_\xi} S_\xi(v)) \nabla S_\xi(v) | d\xi \notag \\
&\leq \dfrac{2}{\epsilon \omega_d} \int_{H_1(0)} J(|\xi|) |F_1''(\sqrt{s_\xi} S_\xi(u)) - F_1''(\sqrt{s_\xi} S_\xi(v))| |\nabla S_\xi(u)| d\xi \notag \\
&\: + \dfrac{2}{\epsilon \omega_d} \int_{H_1(0)} J(|\xi|) |F_1''(\sqrt{s_\xi} S_\xi(v))| |\nabla S_\xi(u) - \nabla S_\xi(v)| d\xi \notag \\
&\leq \dfrac{2 C_3}{\epsilon \omega_d} \int_{H_1(0)} J(|\xi|) \sqrt{s_\xi} |S_\xi(u) - S_\xi(v)| |\nabla S_\xi(u)| d\xi \notag \\
&\: + \dfrac{2C_2}{\epsilon \omega_d} \int_{H_1(0)} J(|\xi|) |\nabla S_\xi(u) - \nabla S_\xi(v)| d\xi \notag \\
&= I_1(x) + I_2(x) \label{eq:bd 3}
\end{align}
where we denote first and second term as $I_1(x)$ and $I_2(x)$. Proceeding similar to (\ref{eq:bd 2}), we can show
\begin{align}
||I_1||^2 &= \int_D \left(\dfrac{2 C_3}{\epsilon \omega_d} \right)^2 \int_{H_1(0)} \int_{H_1(0)} \dfrac{J(|\xi|)}{|\xi|^{3/2}} \dfrac{J(|\eta|)}{|\eta|^{3/2}} |\xi|^{3/2} |\eta|^{3/2} \sqrt{s_\xi} \sqrt{s_\eta} \notag \\
&\: \times |S_\xi(u) - S_\xi(v)| |\nabla S_\xi(u)| |S_\eta(u) - S_\eta(v)| |\nabla S_\eta(u)| d\xi d\eta dx \notag \\
&\leq \int_D \left(\dfrac{2 C_3}{\epsilon \omega_d} \right)^2 \omega_d \bar{J}_{3/2} \int_{H_1(0)} \dfrac{J(|\xi|)}{|\xi|^{3/2}} |\xi|^3 s_\xi |S_\xi(u) - S_\xi(v)|^2 |\nabla S_\xi(u)|^2 d\xi dx. \label{eq:bd 4}
\end{align}
Now
\begin{align*}
& \int_D |S_\xi(u) - S_\xi(v)|^2 |\nabla S_\xi(u)|^2 dx \notag \\
&\leq \dfrac{4 ||u-v||^2_\infty}{\epsilon^2 |\xi|^2} \dfrac{1}{\epsilon^2 |\xi|^2} \int_D 2(|\nabla u(x + \epsilon \xi)|^2 + |\nabla u(x)|^2) dx \notag \\
&\leq \dfrac{16 ||\nabla u||^2 ||u-v||^2_\infty}{\epsilon^4 |\xi|^4}.
\end{align*}
By Sobolev embedding property, see (\ref{eq:sob embedd 1}), we have $||u - v||_\infty \leq C_{e_1} ||u-v||_2$. Thus, we get
\begin{align*}
 \int_D |S_\xi(u) - S_\xi(v)|^2 |\nabla S_\xi(u)|^2 dx \leq \dfrac{16 C_{e_1}^2 ||\nabla u||^2 ||u-v||^2_2}{\epsilon^4 |\xi|^4}.
\end{align*}
Substituting above in (\ref{eq:bd 4}) to get
\begin{align*}
||I_1||^2 &\leq \left(\dfrac{2 C_3}{\epsilon \omega_d} \right)^2 \omega_d \bar{J}_{3/2} \int_{H_1(0)} \dfrac{J(|\xi|)}{|\xi|^{3/2}} |\xi|^3 \epsilon |\xi| \dfrac{16 C_{e_1}^2 ||u||^2_2}{\epsilon^4 |\xi|^4} ||u-v||^2_2 d\xi \notag \\
&= \left(\dfrac{8 C_3 C_{e_1} \bar{J}_{3/2} ||u||_2}{\epsilon^{5/2}} \right)^2 ||u-v||^2_2.
\end{align*}
Let $L_2 = 8 C_3 C_{e_1} \bar{J}_{3/2}$ to write
\begin{align}
||I_1|| &\leq \dfrac{L_2 (||u||_2 + ||v||_2)}{\epsilon^{5/2}} ||u-v||_2. \label{eq:bd 5}
\end{align}

Similarly
\begin{align*}
||I_2||^2 &= \int_D \left( \dfrac{2C_2}{\epsilon \omega_d} \right)^2 \int_{H_1(0)} \int_{H_1(0)} \dfrac{J(|\xi|)}{|\xi|} \dfrac{J(|\eta|)}{|\eta|} |\xi||\eta| \notag \\
&\: \times |\nabla S_\xi(u) - \nabla S_\xi(v)| |\nabla S_\eta(u) - \nabla S_\eta(v)| d\xi d\eta dx \notag \\
&\leq \left( \dfrac{2C_2}{\epsilon \omega_d} \right)^2 \omega_d \bar{J}_1 \int_{H_1(0)} \dfrac{J(|\xi|)}{|\xi|} |\xi|^2 \left[ \int_D |\nabla S_\xi(u) - \nabla S_\xi(v)|^2 dx \right] d\xi.
\end{align*}
This gives
\begin{align}
||I_2|| &\leq \dfrac{4 C_2 \bar{J}_1}{\epsilon^2} ||u-v||_2 = \dfrac{L_1}{\epsilon^2} ||u-v||_2. \label{eq:bd 6}
\end{align}
Thus
\begin{align}
||g_1(u) - g_1(v)|| &\leq \dfrac{\sqrt{\epsilon} L_1 + L_2(||u||_2 + ||v||_2)}{\epsilon^{5/2}} ||u-v||_2. \label{eq:bd 7}
\end{align}

We now work on $|g_2(u) (x) - g_2(v)(x)|$, see (\ref{eq:nabla p u}). Noting the bound on $\nabla \bar{\omega}_\xi$, we get
\begin{align}
& |g_2(u) (x) - g_2(v)(x)| \notag \\
&= \left| \dfrac{2}{\epsilon \omega_d} \int_{H_1(0)} J(|\xi|) \left[ \dfrac{F_1'(\sqrt{s_\xi} S_\xi(u))}{\sqrt{s_\xi}} - \dfrac{F_1'(\sqrt{s_\xi} S_\xi(v))}{\sqrt{s_\xi}} \right] e_\xi \otimes \nabla \bar{\omega}_\xi(s) d\xi \right| \notag \\
&\leq \dfrac{2C_{\omega_1}}{\epsilon \omega_d} \int_{H_1(0)} J(|\xi|) \left| \dfrac{F_1'(\sqrt{s_\xi} S_\xi(u))}{\sqrt{s_\xi}} - \dfrac{F_1'(\sqrt{s_\xi} S_\xi(v))}{\sqrt{s_\xi}} \right| d\xi \notag \\
&\leq \dfrac{2C_{\omega_1}C_2}{\epsilon \omega_d} \int_{H_1(0)} J(|\xi|) |S_\xi(u) - S_\xi(v)| d\xi.
\end{align}
Note that the above inequality is similar to (\ref{eq:bd 1}) and therefore we get 
\begin{align}
||g_2(u) -g_2(v)|| &\leq \dfrac{4C_{\omega_1}C_2 \bar{J}_1}{\epsilon^2} ||u-v||_2 = \dfrac{C_{\omega_1} L_1}{\epsilon^2} ||u-v||_2 \label{eq:bd 8} .
\end{align}
Combining (\ref{eq:bd 7}) and (\ref{eq:bd 8}) to write
\begin{align}
||\nabla P(u) - \nabla P(v)|| &\leq \dfrac{\sqrt{\epsilon} (1+C_{\omega_1})L_1 + L_2(||u||_2 + ||v||_2)}{\epsilon^{5/2}} ||u-v||_2. \label{eq:lipsch 3}
\end{align}

\paragraph{Estimating $||\nabla^2 P(u) - \nabla^2 P(v)||$}From (\ref{eq:nabla 2 p u}), we have
\begin{align}
&||\nabla^2 P(u) - \nabla^2 P(v)|| \notag \\
&\leq ||h_1(u) - h_1(v)|| + ||h_2(u) - h_2(v)|| + ||h_3(u) - h_3(v)|| \notag \\
&\: + ||h_4(u) - h_4(v)|| + ||h_5(u) - h_5(v)||. \label{eq:bd 12}
\end{align}
We can show, using the fact $|\bar{\omega}_\xi(x)| \leq 1$ and $|F_1''(r_1) - F_1''(r_2)| \leq C_3 |r_1 - r_2|$, that
\begin{align}
|h_1(u)(x) - h_1(v)(x)| &\leq \dfrac{2C_3}{\epsilon \omega_d} \int_{H_1(0)} J(|\xi|) \sqrt{s_\xi} |S_\xi(u) - S_\xi(v)| |S_\xi(\nabla^2 u)| d\xi \notag \\
&\: + \dfrac{2C_2}{\epsilon \omega_d} \int_{H_1(0)} J(|\xi|)  |S_\xi(\nabla^2 u) - S_\xi(\nabla^2 v)| d\xi \notag \\
&= I_3(x) + I_4(x). \label{eq:bd 13}
\end{align}
Following similar steps used above we can show 
\begin{align}
||I_3|| &\leq \dfrac{8 C_3 C_{e_1} \bar{J}_{3/2} ||u||_2}{\epsilon^{5/2}} ||u-v||_2 \leq \dfrac{ L_2 (||u||_2 + ||v||_2)}{\epsilon^{5/2}} ||u-v||_2 \label{eq:bd 14}
\end{align}
and
\begin{align}
||I_4|| &\leq \dfrac{4C_2 \bar{J}_1}{\epsilon^2} ||u-v||_2 = \dfrac{L_1}{\epsilon^2} ||u-v||_2, \label{eq:bd 15}
\end{align}
where $L_1 = 4 C_2 \bar{J}_1, L_2 = 8C_3 C_{e_1} \bar{J}_{3/2}$.

Next we focus on $|h_2(u)(x) - h_2(v)(x)|$. We have
\begin{align}
&|h_2(u)(x) - h_2(v)(x)| \notag \\
&\leq \dfrac{2}{\epsilon \omega_d} \int_{H_1(0)} J(|\xi|)\sqrt{s_\xi} |F_1'''(\sqrt{s_\xi} S_\xi(u)) - F_1'''(\sqrt{s_\xi} S_\xi(v))| |S_\xi(\nabla u)|^2 d\xi \notag \\
&\: + \dfrac{2}{\epsilon \omega_d} \int_{H_1(0)} J(|\xi|)\sqrt{s_\xi} |F_1'''(\sqrt{s_\xi} S_\xi(v))| |S_\xi(\nabla u) \otimes S_\xi(\nabla u) - S_\xi(\nabla v) \otimes S_\xi(\nabla v)| d\xi \notag \\
&\leq \dfrac{2C_4}{\epsilon \omega_d} \int_{H_1(0)} J(|\xi|) s_\xi |S_\xi(u) - S_\xi(v)| |S_\xi(\nabla u)|^2 d\xi \notag \\
&\: + \dfrac{2C_3}{\epsilon \omega_d} \int_{H_1(0)} J(|\xi|)\sqrt{s_\xi} |S_\xi(\nabla u) \otimes S_\xi(\nabla u) - S_\xi(\nabla v) \otimes S_\xi(\nabla v)| d\xi \notag \\
&= I_5(x) + I_6(x). \label{eq:bd 16}
\end{align}
Proceeding as below for $||I_5||^2$ 
\begin{align}
&||I_5||^2  \notag \\
&\leq \int_D \left( \dfrac{2C_4}{\epsilon \omega_d} \right)^2 \int_{H_1(0)} \int_{H_1(0)} \dfrac{J(|\xi|)}{|\xi|^2} \dfrac{J(|\eta|)}{|\eta|^2} |\xi|^2 s_\xi |\eta|^2 s_\eta \notag \\
&\: \times |S_\xi(u) - S_\xi(v)| |S_\xi(\nabla u)|^2 |S_\eta(u) - S_\eta(v)| |S_\eta(\nabla u)|^2 d\xi d\eta dx \notag \\
&\leq \int_D \left( \dfrac{2C_4}{\epsilon \omega_d} \right)^2 \omega_d \bar{J}_2 \int_{H_1(0)} \dfrac{J(|\xi|)}{|\xi|^2} |\xi|^4 s_\xi^2 |S_\xi(u) - S_\xi(v)|^2 |S_\xi(\nabla u)|^4 d\xi dx \notag \\
&\leq \left( \dfrac{2C_4}{\epsilon \omega_d} \right)^2 \omega_d \bar{J}_2 \int_{H_1(0)} \dfrac{J(|\xi|)}{|\xi|^2} |\xi|^4 s_\xi^2 \dfrac{4 ||u-v||^2_\infty}{\epsilon^2 |\xi|^2} \left[ \int_D |S_\xi(\nabla u)|^4 dx \right] d\xi. \label{eq:bd 16.1}
\end{align}

We estimate the term in square bracket. Using the identity $(|a| + |b|)^4 \leq (2 |a|^2 + 2|b|^2)^2 \leq 8 |a|^4 + 8|b|^4$, we have
\begin{align}
\int_D |S_\xi(\nabla u)|^4 dx &\leq \dfrac{8}{\epsilon^4 |\xi|^4} \int_D (|\nabla u(x+ \epsilon \xi)|^4 + |\nabla u(x)|^4) dx \notag \\
&\leq \dfrac{16}{\epsilon^4 |\xi|^4} ||\nabla u||^4_{L^4(D;\R^{d\times d})}.
\end{align}
where $||u||_{L^4(D,\R^d)} = \left[\int_D |u|^4 dx \right]^{1/4}$. Using the Sobolev embedding property of $u\in H^2(D;\R^d)$ as mentioned in (\ref{eq:sob embedd 2}), we get
\begin{align}
\int_D |S_\xi(\nabla u)|^4 dx &\leq \dfrac{16}{\epsilon^4 |\xi|^4} C_{e_2}^4 ||\nabla u||^4_{H^1(D;\R^{d\times d})} \leq \dfrac{16 C^4_{e_2}}{\epsilon^4 |\xi|^4} ||u||^4_2. \label{eq:bd 17}
\end{align}
Using $||u - v||_\infty \leq C_{e_1}||u-v||_2$ and above estimate in (\ref{eq:bd 16.1}) to have
\begin{align*}
||I_5||^2 &\leq \left( \dfrac{2C_4}{\epsilon \omega_d} \right)^2 \omega_d \bar{J}_2 \int_{H_1(0)} \dfrac{J(|\xi|)}{|\xi|^2} |\xi|^4 s_\xi^2 \dfrac{4 C^2_{e_1} ||u-v||^2_2}{\epsilon^2 |\xi|^2} \dfrac{16 C^4_{e_2}}{\epsilon^4 |\xi|^4} ||u||^4_2 d\xi,
\end{align*}
and we obtain
\begin{align}
||I_5|| &\leq \dfrac{16 C_4 C_{e_1} C^2_{e_2} \bar{J}_2 ||u||^2_2}{\epsilon^3} ||u-v||_2 \leq \dfrac{L_3 (||u||_2+||v||_2)^2}{\epsilon^3} ||u-v||_2 \label{eq:bd 18}
\end{align}
where we let $L_3 = 16 C_4 C_{e_1} C^2_{e_2} \bar{J}_2$.

Next we use
\begin{align*}
|S_\xi(\nabla u) \otimes S_\xi(\nabla u) - S_\xi(\nabla v) \otimes S_\xi(\nabla v)| \leq (|S_\xi(\nabla u)| + |S_\xi(\nabla v)|) |S_\xi(\nabla u) - S_\xi(\nabla v)|
\end{align*}
to estimate $||I_6||$ as follows
\begin{align}
&||I_6||^2 \notag \\
&\leq \int_D \left( \dfrac{2C_3}{\epsilon \omega_d} \right)^2 \int_{H_1(0)} \int_{H_1(0)} \dfrac{J(|\xi|)}{|\xi|^{3/2}} \dfrac{J(|\eta|)}{|\eta|^{3/2}} |\xi|^{3/2} |\eta|^{3/2}\sqrt{s_\xi s_\eta} \notag \\
&\: \times (|S_\xi(\nabla u)| + |S_\xi(\nabla v)|) |S_\xi(\nabla u) - S_\xi(\nabla v)| \notag \\
&\: \times (|S_\eta(\nabla u)| + |S_\eta(\nabla v)|) |S_\eta(\nabla u) - S_\eta(\nabla v)| d\xi d\eta dx \notag \\
&\leq \int_D \left( \dfrac{2C_3}{\epsilon \omega_d} \right)^2 \omega_d \bar{J}_{3/2} \int_{H_1(0)} \dfrac{J(|\xi|)}{|\xi|^{3/2}} |\xi|^3 \epsilon |\xi| (|S_\xi(\nabla u)| + |S_\xi(\nabla v)|)^2 |S_\xi(\nabla u) - S_\xi(\nabla v)|^2 d\xi dx \notag \\
&= \left( \dfrac{2C_3}{\epsilon \omega_d} \right)^2 \omega_d \bar{J}_{3/2} \int_{H_1(0)} \dfrac{J(|\xi|)}{|\xi|^{3/2}} |\xi|^3 \epsilon |\xi| \left[ \int_D (|S_\xi(\nabla u)| + |S_\xi(\nabla v)|)^2 |S_\xi(\nabla u) - S_\xi(\nabla v)|^2 dx \right] d\xi. \label{eq:bd 19}
\end{align}
We focus on the term in square bracket. Using the H\"older inequality, we have
\begin{align}
& \int_D (|S_\xi(\nabla u)| + |S_\xi(\nabla v)|)^2 |S_\xi(\nabla u) - S_\xi(\nabla v)|^2 dx \notag \\
&\leq \left( \int_D (|S_\xi(\nabla u)| + |S_\xi(\nabla v)|)^4 dx \right)^{1/2} \left( \int_D |S_\xi(\nabla u) - S_\xi(\nabla v)|^4 dx \right)^{1/2} .\label{eq:bd 20}
\end{align}
Using $(|a| + |b|)^4 \leq 8 |a|^4 + 8|b|^4$, we get
\begin{align}
&\int_D (|S_\xi(\nabla u)| + |S_\xi(\nabla v)|)^4 dx \notag \\
&\leq 8 \left[ \int_D |S_\xi(\nabla u)|^4 dx + \int_D |S_\xi(\nabla v)|^4 dx \right] \notag \\
&\leq 8 \left[ \dfrac{8}{\epsilon^4 |\xi|^4} \int_D (|\nabla u(x+\epsilon \xi)|^4 + |\nabla u(x)|^4) dx +\dfrac{8}{\epsilon^4 |\xi|^4} \int_D (|\nabla v(x+\epsilon \xi)|^4 + |\nabla v(x)|^4) dx  \right] \notag \\
&\leq \dfrac{128}{\epsilon^4 |\xi|^4} (||\nabla u||^4_{L^4(D;\R^{d\times d})} + ||\nabla v||^4_{L^4(D;\R^{d\times d})} ) \notag \\
&\leq \dfrac{128 C^4_{e_2}}{\epsilon^4 |\xi|^4} (||\nabla u||^4_{H^1(D;\R^{d\times d})} + ||\nabla v||^4_{H^1(D;\R^{d\times d})} ) \notag \\
&\leq \dfrac{128 C^4_{e_2}}{\epsilon^4 |\xi|^4} (||u||^4_2 + ||v||^4_2) \notag \\
&\leq \dfrac{128 C^4_{e_2}}{\epsilon^4 |\xi|^4} (||u||_2 + ||v||_2)^4. \label{eq:bd 21}
\end{align}
where we used Sobolev embedding property (\ref{eq:sob embedd 2}) in third last step. Proceeding similarly to get
\begin{align}
& \int_D |S_\xi(\nabla u) - S_\xi(\nabla v)|^4 dx \notag \\
&\leq \dfrac{8}{\epsilon^4 |\xi|^4} \left[ \int_D |\nabla (u-v)(x+\epsilon \xi)|^4 dx + \int_D |\nabla (u-v)(x)|^4 dx \right] \notag \\
&\leq \dfrac{16}{\epsilon^4 |\xi|^4} ||\nabla (u-v)||^4_{L^4(D,\R^{d\times d})} \notag \\
&\leq \dfrac{16 C^4_{e_2}}{\epsilon^4 |\xi|^4} ||u-v||^4_2. \label{eq:bd 22}
\end{align}
Substituting (\ref{eq:bd 21}) and (\ref{eq:bd 22}) into (\ref{eq:bd 20}) to get
\begin{align*}
& \int_D (|S_\xi(\nabla u)| + |S_\xi(\nabla v)|)^2 |S_\xi(\nabla u) - S_\xi(\nabla v)|^2 dx \notag \\
&\leq \left( \dfrac{128 C^4_{e_2}}{\epsilon^4 |\xi|^4} (||u||_2 + ||v||_2 )^4 \right)^{1/2} \left( \dfrac{16 C^4_{e_2}}{\epsilon^4 |\xi|^4} ||u-v||^4_2\right)^{1/2} \notag \\
&= \dfrac{32 \sqrt{2} C^4_{e_2}}{\epsilon^4 |\xi|^4} (||u||_2 +||v||_2)^2 ||u-v||_2^2 \notag \\
&\leq \dfrac{64 C^4_{e_2}}{\epsilon^4 |\xi|^4} (||u||_2 +||v||_2)^2 ||u-v||_2^2 .
\end{align*}
Substituting above in (\ref{eq:bd 19}) to get
\begin{align*}
&||I_6||^2 \notag \\ 
& \leq \left( \dfrac{2C_3}{\epsilon \omega_d} \right)^2 \omega_d \bar{J}_{3/2} \int_{H_1(0)} \dfrac{J(|\xi|)}{|\xi|^{3/2}} |\xi|^3 \epsilon |\xi| \left[ \dfrac{64 C^4_{e_2}}{\epsilon^4 |\xi|^4} (||u||_2 +||v||_2)^2 ||u-v||_2^2 \right] d\xi.
\end{align*}
From above we have
\begin{align}
||I_6|| &\leq \dfrac{16 C_3 C_{e_2}^2 \bar{J}_{3/2} (||u||_2 + ||v||_2)}{\epsilon^{5/2}} ||u-v||_2 = \dfrac{L_4 (||u||_2 + ||v||_2)}{\epsilon^{5/2}} ||u-v||_2, \label{eq:bd 23}
\end{align}
where we let $L_4 = 16 C_3 C_{e_2}^2 \bar{J}_{3/2}$.

From the expression of $h_3(u)(x)$ and $h_5(u)(x)$ we find that it is similar to term $g_1(u)(x)$ from the point of view of $L^2$ norm. Also, $h_4(u)(x)$ is similar to $P(u)(x)$. We easily have
\begin{align*}
|h_4(u)(x) - h_4(v)(x)| &\leq  \dfrac{2 C_2 C_{\omega_2}}{\epsilon \omega_d} \int_{H_1(0)} J(|\xi|) |S_\xi(u) - S_\xi(v)| d\xi,
\end{align*}
where we used the fact that $|\nabla^2 \bar{\omega}_\xi(x)| \leq C_{\omega_2}$. Above is similar to the bound on $|P(u)(x) - P(v)(x)|$, see (\ref{eq:bd 1}), therefore we have
\begin{align}
||h_4(u) - h_4(v)|| &\leq \dfrac{L_1 C_{\omega_2}}{\epsilon^2} ||u-v||_2. \label{eq:bd 24}
\end{align}
Similarly, we have
\begin{align}
&|h_3(u)(x) - h_3(v)(x)| \notag \\
&\leq \dfrac{2}{\epsilon \omega_d} \int_{H_1(0)} J(|\xi|) |F_1''(\sqrt{s_\xi} S_\xi(u)) - F_1''(\sqrt{s_\xi} S_\xi(v))| |\nabla \bar{\omega}_\xi(x)| |S_\xi(\nabla u)| d\xi \notag \\
&\: + \dfrac{2}{\epsilon \omega_d} \int_{H_1(0)} J(|\xi|) |F_1''(\sqrt{s_\xi} S_\xi(v))| |e_\xi \otimes \nabla \bar{\omega}_\xi(x) \otimes S_\xi(\nabla u) - e_\xi \otimes \nabla \bar{\omega}_\xi(x) \otimes S_\xi(\nabla v)| d\xi \notag \\
&\leq \dfrac{2 C_3 C_{\omega_1}}{\epsilon \omega_d} \int_{H_1(0)} J(|\xi|) \sqrt{s_\xi} |S_\xi(u) - S_\xi(v)| |S_\xi(\nabla u)| d\xi \notag \\
&\: + \dfrac{2 C_2 C_{\omega_1}}{\epsilon \omega_d} \int_{H_1(0)} J(|\xi|) |S_\xi(\nabla u) - S_\xi(\nabla v)| d\xi \notag \\
&= C_{\omega_1} (I_1(x) + I_2(x)),
\end{align}
where $I_1(x)$ and $I_2(x)$ are given by (\ref{eq:bd 3}). From (\ref{eq:bd 5}) and (\ref{eq:bd 6}), have
\begin{align}
||h_3(u) - h_3(v)|| &\leq C_{\omega_1} (||I_1|| + ||I_2||) \notag \\
&\leq \dfrac{\sqrt{\epsilon} C_{\omega_1} L_1 + C_{\omega_1} L_2(||u||_2 + ||v||_2)}{\epsilon^{5/2}} ||u-v||_2. \label{eq:bd 25}
\end{align}
Expression of $h_3(u)$ and $h_5(u)$ is similar and hence we have
\begin{align}
||h_5(u) - h_5(v)|| &\leq C_{\omega_1} (||I_1|| + ||I_2||) \notag \\
&\leq \dfrac{\sqrt{\epsilon} C_{\omega_1} L_1 + C_{\omega_1} L_2(||u||_2 + ||v||_2)}{\epsilon^{5/2}} ||u-v||_2. \label{eq:bd 26}
\end{align}

We collect results to get
\begin{align}
&||\nabla^2 P(u) - \nabla^2 P(v)|| \notag \\
&\leq \left[ \dfrac{ \epsilon L_1 + \sqrt{\epsilon} L_2 (||u||_2 + ||v||_2) + L_3 (||u||_2 + ||v||_2)^2 + \sqrt{\epsilon} L_4 (||u||_2 + ||v||_2) }{\epsilon^3} \right. \notag \\
&\quad \left.+ \dfrac{\epsilon C_{\omega_2} L_1 + 2\epsilon C_{\omega_1} L_1 +  2\sqrt{\epsilon}C_{\omega_1} L_2(||u||_2 + ||v||_2) }{\epsilon^3} \right] ||u-v||_2 \notag \\
&\leq \left[ \dfrac{\epsilon(1+ 2C_{\omega_1} + C_{\omega_2}) L_1 + \sqrt{\epsilon}(L_2 + 2C_{\omega_1}L_2 + L_4)(||u||_2 + ||v||_2)}{\epsilon^3} \right. \notag \\
&\quad \left. + \dfrac{L_3(||u||_2 + ||v||_2)^2 }{\epsilon^3} \right] ||u-v||_2. \label{eq:lipsch 4}
\end{align}
%

We now combine (\ref{eq:lipsch 2}), (\ref{eq:lipsch 3}), and (\ref{eq:lipsch 4}), to get
\begin{align}
&||P(u) - P(v)||_2 \notag \\
&\leq \left[ \dfrac{2 \epsilon L_1 + \epsilon(1+C_{\omega_1})L_1 + \sqrt{\epsilon}(||u||_2 + ||v||_2)}{\epsilon^3} \right. \notag \\
&\quad +\dfrac{\epsilon(1+ 2C_{\omega_1} + C_{\omega_2}) L_1 + \sqrt{\epsilon}(L_2 + 2C_{\omega_1}L_2 + L_4)(||u||_2 + ||v||_2)}{\epsilon^3} \notag \\
&\quad \left. + \dfrac{L_3(||u||_2 + ||v||_2)^2 }{\epsilon^3} \right] ||u-v||_2.
\end{align}
We let
\begin{align}
\bar{L}_1 := (4 + 3C_{\omega_1} + C_{\omega_2})L_1, \: \bar{L}_2 := (1+ 2C_{\omega_1})L_2 + L_4, \: \bar{L}_3 := L_3  \label{eq:lipsch const def H2}
\end{align}
to write
\begin{align}
&||P(u) - P(v)||_2 \notag \\
&\leq \dfrac{\bar{L}_1 + \bar{L}_2 (||u||_2 + ||v||_2) + \bar{L}_3(||u||_2 + ||v||_2)^2}{\epsilon^3} ||u-v||_2
\end{align}
and this completes the proof of (\ref{eq:lip of f in H2}).

We now bound the peridynamic force. Note that $F_1'(0) = 0$, and $S_\xi(v) = 0$ if $v=0$. Substituting $v=0$ in (\ref{eq:lipsch 2}) to get
\begin{align}
||P(u)||&\leq \dfrac{L_1}{\epsilon^2} ||u||_2. \label{eq:bd f 1}
\end{align}
For $||g_1(u)||$ and $||g_2(u)||$ we proceed differently. For $||g_2(u)||$, we substitute $v=0$ in (\ref{eq:bd 8}) to get
\begin{align}
||g_2(u)|| &\leq \dfrac{C_{\omega_1} L_1}{\epsilon^2} ||u||_2. \label{eq:bd f 2}
\end{align}
For $||g_1(u)||$, we proceed as follows
\begin{align}
|g_1(u)(x)| &\leq \dfrac{2C_2}{\epsilon \omega_d} \int_{H_1(0)} J(|\xi|) |\nabla S_\xi(u)| d\xi \notag \\
&\leq \dfrac{2C_2}{\epsilon^2 \omega_d} \int_{H_1(0)} \dfrac{J(|\xi|)}{|\xi|} (|\nabla u(x+\epsilon\xi)| + |\nabla u(x)|) d\xi,
\end{align}
and we have
\begin{align}
||g_1(u)||^2 &\leq \left( \dfrac{2C_2}{\epsilon^2 \omega_d} \right)^2 \omega_d \bar{J}_1 \int_{H_1(0)} \dfrac{J(|\xi|)}{|\xi|} \left[ \int_D (|\nabla u(x+\epsilon\xi)| + |\nabla u(x)|)^2 dx \right] d\xi \notag \\
&\leq \left( \dfrac{4 C_2 \bar{J}_1}{\epsilon^2} \right)^2 ||\nabla u||^2,
\end{align}
i.e.
\begin{align}
||g_1(u)|| &\leq \dfrac{L_1}{\epsilon^2} ||u||_2. \label{eq:bd f 3}
\end{align}
Combining (\ref{eq:bd f 2}) and (\ref{eq:bd f 3}) gives
\begin{align}
||\nabla P(u)|| &\leq \dfrac{(1 + C_{\omega_1}) L_1}{\epsilon^2} ||u||_2. \label{eq:bd f 4}
\end{align}

We now estimate $||\nabla^2 P(u)||$ from above. We have from (\ref{eq:nabla 2 p u})
\begin{align*}
||\nabla^2 P(u)|| &\leq ||h_1(u)|| + ||h_2(u)|| + ||h_3(u)|| + ||h_4(u)||+ ||h_5(u)||.
\end{align*}
From expression of $h_1(u)$ we find that
\begin{align*}
||h_1(u)|| &\leq \dfrac{4C_2 \bar{J}_1}{\epsilon^2} ||u||_2 = \dfrac{L_1}{\epsilon^2} ||u||_2
\end{align*}
Bound on $||h_2(u)||$ is similar to $I_6$, see (\ref{eq:bd 16}), and we have
\begin{align*}
||h_2(u)|| \leq \dfrac{8C_3 C^2_{e_2} \bar{J}_{3/2}}{\epsilon^{5/2}} ||u||^2_2 \leq \dfrac{L_4}{\epsilon^{5/2}} ||u||_2^2,
\end{align*}
where $L_4 = 16C_3 C^2_{e_2} \bar{J}_{3/2}$. Case of $||h_3(u)||$ and $||h_5(u)||$ is similar to $||g_1(u)||$, and case of $||h_4(u)||$ is similar to $||P(u)||$.  We thus have
\begin{align*}
||h_4(u)|| &\leq \dfrac{C_{\omega_2} L_1}{\epsilon^2} ||u||_2 
\end{align*}
and
\begin{align*}
||h_3(u)|| \leq \dfrac{C_{\omega_1}L_1}{\epsilon^2} ||u||_2 \quad \text{ and } \quad ||h_5(u)|| \leq \dfrac{C_{\omega_1}L_1}{\epsilon^2} ||u||_2.
\end{align*}
Combining above to get
\begin{align}
||\nabla^2 P(u)|| &\leq \dfrac{\sqrt{\epsilon}(1+ C_{\omega_2} + 2C_{\omega_1})L_1 + L_4 ||u||_2}{\epsilon^{5/2}} ||u||_2 \label{eq:bd f 5}.
\end{align}
Finally, after combining (\ref{eq:bd f 1}), (\ref{eq:bd f 4}), and (\ref{eq:bd f 5}), we get
\begin{align*}
||P(u)||_2 &\leq \dfrac{\sqrt{\epsilon}(4 + 3C_{\omega_1} + C_{\omega_2})L_1 + L_4||u||_2 }{\epsilon^{5/2}} ||u||_2.
\end{align*}
We let
\begin{align}\label{eq:lipsch const def H2 2}
\bar{L}_4 := \bar{L}_1 \qquad \text{and}\qquad \bar{L}_5 := L_4
\end{align}
to write
\begin{align}
||P(u)||_2 &\leq \dfrac{\bar{L}_4||u||_2 + \bar{L}_5||u||^2_2}{\epsilon^{5/2}}.
\end{align}
This completes the proof of (\ref{eq:bound on f in H2}).

\subsection{Local and global existence of solution in $H^2 \cap H^1_0$ space}
\label{section:existence}
In this section, we prove Theorem \ref{thm:existence over finite time domain}. We first prove local existence for a finite time interval. We find that we can choose this time interval  independent of the initial data. We repeat the local existence theorem to uniquely continue the local solution over any finite time interval. 
The existence and uniqueness of local solutions is stated in the following theorem.

\begin{theorem}\label{thm:local existence}
\textbf{Local existence and uniqueness} \\
Given $X= W\times W$, $b(t)\in W$, and  initial data $x_0=(u_0,v_0)\in X$. We suppose that $b(t)$ is continuous in time over some time interval $I_0=(-T,T)$ and satisfies $\sup_{t\in I_0} ||b(t)||_2 < \infty$.  Then, there exists a time interval $I'=(-T',T')\subset I_0$ and unique solution $y =(y^1,y^2) $ such that $y\in C^1(I';X)$ and

\begin{equation}
y(t)=x_0+\int_0^tF^\epsilon(y(\tau),\tau)\,d\tau,\hbox{  for $t\in I'$}
\label{8loc}
\end{equation}

or equivalently

\begin{equation*}
y'(t)=F^\epsilon(y(t),t),\hbox{with    $y(0)=x_0$},\hbox{  for $t\in I'$}
\label{11loc}
\end{equation*}
where $y(t)$ and $y'(t)$ are Lipschitz continuous in time for $t\in I'\subset I_0$.
\end{theorem}

\begin{proof}
To prove Theorem \ref{thm:local existence}, we proceed as follows. 
Write $y(t)=(y^1(t),y^2(t))$ with $||y||_X=||y^1(t)||_2+||y^2(t)||_2$. 
Let us consider $R>||x_0||_X$ and define the ball $B(0,R)= \{y\in X:\, ||y||_X<R\}$. 
Let $r < \min \{1, R - ||x_0||_X\}$. We clearly have $r^2 < (R-||x_0||_X)^2$ as well as $r^2 < r < R - ||x_0||_X$. Consider the ball $B(x_0,r^2)$ defined by
\begin{equation}
B(x_0,r^2)=\{y\in X:\,||y-x_0||_X<r^2\}.
\label{balls}
\end{equation}
Then we have $B(x_0, r^2) \subset B(x_0, r) \subset B(0, R)$, see Fig \ref{figurenested}. 

To recover the existence and uniqueness we introduce the transformation
\begin{equation*}
S_{x_0}(y)(t)=x_0+\int_0^tF^\epsilon(y(\tau),\tau)\,d\tau.
\label{0}
\end{equation*}
Introduce $0<T'<T$ and the associated set $Y(T')$ of functions in $W$ taking values in $B(x_0,r^2)$ for $I'=(-T',T')\subset I_0=(-T,T)$. 
The goal is to find appropriate interval $I'=(-T',T')$ for which $S_{x_0}$ maps into the corresponding set $Y(T')$. Writing out the transformation with $y(t)\in Y(T')$ gives
\begin{eqnarray}
&&S_{x_0}^1(y)(t)=x_0^1+\int_0^t y^2(\tau)\,d\tau\label{1}\\
&&S_{x_0}^2(y)(t)=x_0^2+\int_0^t(-\nabla PD^\epsilon(y^1(\tau))+b(\tau))\,d\tau.\label{2}
\end{eqnarray}
We have from (\ref{1})
\begin{eqnarray}
||S_{x_0}^1(y)(t)-x_0^1||_2\leq\sup_{t\in(-T',T')}||y^2(t)||_2 T'\label{4}.
\end{eqnarray}
Using bound on $-\nabla PD^\epsilon$ in Theorem \ref{thm:lip in H2}, we have from (\ref{2})
\begin{eqnarray}
||S_{x_0}^2(y)(t)-x_0^2||_2\leq \int_0^t \left[ \dfrac{\bar{L}_4}{\epsilon^{5/2}} ||y^1(\tau)||_2 + \dfrac{\bar{L}_5}{\epsilon^{5/2}} ||y^1(\tau)||^2_2 + ||b(\tau)||_2 \right] d\tau. \label{3}
\end{eqnarray}
Let $\bar{b} = \sup_{t\in I_0} ||b(t)||_2$. Noting that transformation $S_{x_0}$ is defined for $t\in I' = (-T', T')$ and $y(\tau) = (y^1(\tau), y^2(\tau)) \in B(x_0, r^2) \subset B(0,R)$ as $y \in Y(T')$, we have from (\ref{3}) and (\ref{4})
\begin{align*}
||S_{x_0}^1(y)(t)-x_0^1||_2 &\leq R T' , \\
||S_{x_0}^2(y)(t)-x_0^2||_2 &\leq \left[ \dfrac{\bar{L}_4 R + \bar{L}_5 R^2}{\epsilon^{5/2}} + \bar{b} \right] T'.
\end{align*}
Adding these inequalities delivers
\begin{align}
||S_{x_0}(y)(t)-x_0||_X\leq \left[ \dfrac{\bar{L}_4 R + \bar{L}_5 R^2}{\epsilon^{5/2}} + R + \bar{b} \right] T'. \label{5}
\end{align}
Choosing $T'$ as below 
\begin{align}
T' < \dfrac{r^2}{\left[ \frac{\bar{L}_4 R + \bar{L}_5 R^2}{\epsilon^{5/2}} + R + \bar{b} \right]}
\end{align}
will result in $S_{x_0}(y)\in Y(T')$ for all $y\in Y(T')$ as
\begin{align}
||S_{x_0}(y)(t)-x_0||_X < r^2.
\end{align}

Since $r^2 < (R-||x_0||_X)^2$, we have
\begin{align*}
T' < \dfrac{r^2}{\left[ \frac{\bar{L}_4 R + \bar{L}_5 R^2}{\epsilon^{5/2}} + R + \bar{b} \right]} < \dfrac{(R-||x_0||_X)^2}{\left[ \frac{\bar{L}_4 R + \bar{L}_5 R^2}{\epsilon^{5/2}} + R + \bar{b} \right]}.
\end{align*}
Let $\theta(R)$ be given by
\begin{align}
\theta(R) := \dfrac{(R-||x_0||_X)^2}{\left[ \frac{\bar{L}_4 R + \bar{L}_5 R^2}{\epsilon^{5/2}} + R + \bar{b} \right]}.
\end{align}
$\theta(R)$ is increasing with $R>0$ and satisfies
\begin{align}
\theta_\infty := \lim_{R\to \infty} \theta(R) = \dfrac{\epsilon^{5/2}}{\bar{L}_5}.
\end{align}

So given $R$ and $||x_0||_X$ we choose $T'$ according to 
\begin{align}
\dfrac{\theta(R)}{2} < T' < \theta(R),
\end{align}
and set $I' = (-T', T')$. This way we have shown that for time domain $I'$ the transformation $S_{x_0}(y)(t)$ as defined in \eqref{8loc} maps $Y(T')$ into itself. The required Lipschitz continuity follows from \eqref{eq:lip of f in H2} and existence and uniqueness of solution can be established using an obvious modification of [Theorem VII.3, \cite{MA-Brezis}].
\end{proof}

We now prove Theorem \ref{thm:existence over finite time domain}. From the proof of Theorem \ref{thm:local existence} above, we have a unique local solution over a time domain $(-T',T')$ with $\frac{\theta(R)}{2}<T'$. Since $\theta(R)\nearrow \epsilon^{5/2}/\bar{L}_5$ as $R\nearrow \infty$ we can fix a tolerance $\eta>0$ so that $[(\epsilon^{5/2}/2\bar{L}_5)-\eta]>0$. Then for any initial condition in $W$ and $b=\sup_{t\in[-T,T)}||b(t)||_2$ we can choose $R$ sufficiently large so that $||x_0||_X<R$ and  $0<(\epsilon^{5/2}/2\bar{L}_5)-\eta<T'$. Since choice of $T'$ is independent of initial condition and $R$, we can always find local solutions for time intervals $(-T',T')$ for $T'$ larger than $[(\epsilon^{5/2}/2\bar{L}_5)-\eta]>0$. Therefore we apply the local existence and uniqueness result to uniquely continue local solutions up to an arbitrary time interval $(-T,T)$.

\begin{figure} 
\centering
\begin{tikzpicture}[xscale=0.8,yscale=0.8]
\draw [] (0.0,0.0) circle [radius=3.5];
\draw [thick,dashed] (0.7,1.8) circle [radius=1.5];
\draw [] (0.7,1.8) circle [radius=0.6];
\draw [->,thick] (0.7,1.8) -- (1.1,2.1);
\node [right] at (0.5,1.6) {$x_0$};
\node [below] at (-1.4,1.52) {$B(x_0,r)$};
\node [below] at (0.35,3.1) {$B(x_0,r^2)$};
\node [below] at (0.0,0.0) {$0$};
\draw [->,thick] (0,0) -- (3.5,0.0);
\node [below] at (1.8,0.0) {$R$};
\node [below] at (-2.5,-2.8) {$B(0,R)$};
\end{tikzpicture} 
\caption{Geometry.}
 \label{figurenested}
\end{figure}
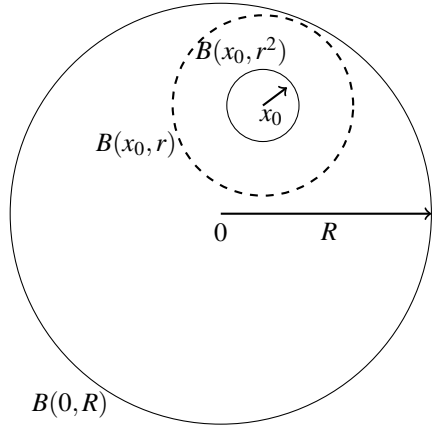

\subsection{Proof of the higher regularity with respect to time}

In this section we prove that the peridynamic evolutions have higher regularity in time for body forces that that are differentiable in time.
To see this we take the time derivative of \eqref{eq:per equation} to get a second order differential equation in time for $v=\dot{u}$ given by
\begin{align}\label{eq:per equation v}
\rho \partial^2_{tt} v(t,x) &= Q(v(t); u(t))(x) + \dot{b}(t,x),
\end{align}
where $Q(v;u)$ is an operator that depends on the solution $u$ of  \eqref{eq:per equation} and acts on $v$. It is given by, $\forall x\in D$,
\begin{align}\label{eq:def Q}
Q(v;u)(x) &:= \dfrac{2}{\epsilon^d \omega_d} \int_{H_{\epsilon}(x)} \partial^2_S W^{\epsilon}(S(y,x;u),y - x) S(y,x;v) \dfrac{y - x}{|y - x|} dy.
\end{align}
Clearly, for given $u$, $Q(v;u)$ acts linearly on $v$ which implies that the equation for $v$ (\ref{eq:per equation v}) is a linear nonlocal equation. The linearity of $Q(v;u)$ implies Lipschitz continuity for $v \in W$ as stated below.

\begin{theorem}\label{thm:Q lip in H2}
\textbf{Lipschitz continuity of $Q$}\\
Let $u \in W$ be any given field. Then for all $v,w \in W$, we have
\begin{align}
||Q(v;u) - Q(w;u)||_2 &\leq \dfrac{L_6(1 + ||u||_2 + ||u||^2_2)}{\epsilon^3} ||v-w||_2\label{eq:lip of Q in H2}
\end{align}
where constant $L_6$ does not depend on $u,v,w$. This gives for all $v\in W$ the upper bound,
\begin{align}
||Q(v;u)||_2 &\leq \dfrac{L_6(1 + ||u||_2 + ||u||^2_2)}{\epsilon^3} ||v||_2 \label{eq:bound on Q in H2}.
\end{align}
\end{theorem}
We postpone the proof of Theorem \ref{thm:Q lip in H2} and provide it in the following subsection.
From the theorem, we see that if $u$ is a solution of (\ref{eq:per equation}) and $u \in C^2(I_0;W)$ then we have for all $t \in I_0$ the inequality
\begin{align}\label{eq:bound on Q in H2 time}
||Q(v;u(t))||_2 &\leq \dfrac{L_6(1 + \sup_{s \in I_0} ||u(s)||_2 + \sup_{s\in I_0} ||u(s)||^2_2)}{\epsilon^3} ||v||_2.
\end{align}
Next we remark that the Lipschitz continuity of $y'(t)$ stated in Theorem \ref{thm:existence over finite time domain} implies $\lim_{t\rightarrow 0^{\pm}}\partial^2_{tt} u(t,x)=\partial^2_{tt}u(0,x)$.
We now show that $v(t,x)=\partial_t u(t,x)$ is the unique solution of the following initial boundary value problem.

\begin{theorem}\label{thm:soln of v equation} 
\textbf{Initial value problem for $v(t,x)$}\\
Suppose the initial data and righthand side $b(t)$ satisfy the hypothesis of Theorem \ref{thm:existence over finite time domain} and we suppose further that  $\dot{b}(t)$ exists and is continuous in time for $t\in I_0$ and $\sup_{t\in I_0} ||\dot{b}(t)||_2 < \infty$. Then $v(t,x)$ is the unique solution to the initial value problem $v(0,x)=v_0(x)$,  $\partial_t v(0,x)=\partial^2_{tt}u(0,x)$,
\begin{align}\label{eq:per equation vv}
\rho \partial^2_{tt} v(t,x) &= Q(v(t); u(t))(x) + \dot{b}(t,x), \hbox{  $t\in I_0$},
\end{align}
$v \in C^2(I_0; W)$ and 
\begin{align}\label{eq:per equation vsubtt estt}
|| \partial^3_{tt} v(t,x)||_2 &\leq ||Q(v(t); u(t))(x)||_2 + ||\dot{b}(t,x)||_2.
\end{align}
\end{theorem}
Theorem \ref{thm:higher regularity} now follows immediately from Theorem \ref{thm:soln of v equation} noting that $\partial_t u(t,x)=v(t,x)$
together with \eqref {eq:bound on Q in H2 time} and \eqref{eq:per equation vsubtt estt} .
The proof of Theorem \ref{thm:soln of v equation}  follows from the Lipschitz continuity (\ref{eq:bound on Q in H2 time}) and the Banach fixed point theorem as in \cite{MA-Brezis}. 
\subsection{Lipschitz continuity of $Q(v;u)$ in $ H^2\cap H^1_0$}
We conclude by explicitly establishing the Lipschitz continuity of $Q(v;u)$.
Recall that $Q(v; u)$ is given by 
\begin{align*}
Q(v;u)(x) &= \dfrac{2}{\epsilon^d \omega_d} \int_{H_{\epsilon}(x)} \partial^2_S W^{\epsilon}(S(y,x;u),y - x) S(y,x;v) \dfrac{y - x}{|y - x|} dy.
\end{align*}
From expression of $W^\epsilon$ in (\ref{eq:per pot}) and using the notation $F_1(r) = f(r^2)$ we have
\begin{align*}
\partial^2_S W^\epsilon(S, y-x) &= \partial^2_S \left( \omega(x) \omega(y) \frac{J^\epsilon(|y-x|)}{\epsilon |y-x|} F_1(\sqrt{|y-x|} S) \right) \notag \\
&= \omega(x) \omega(y) \frac{J^\epsilon(|y-x|)}{\epsilon} F''_1(\sqrt{|y-x|} S).
\end{align*}
Substituting above, using the change of variable $y = x+ \epsilon\xi, \xi \in H_1(0)$ and using the notation of previous subsections, we get
\begin{align}\label{eq:Q simple}
Q(v;u)(x) &= \dfrac{2}{\epsilon \omega_d} \int_{H_1(0)} \bar{\omega}_\xi(x) J(|\xi|) F''_1(\sqrt{s_\xi} S_\xi(u)) S_\xi(v) e_\xi d\xi.
\end{align}

We study $||Q(v;u) - Q(w;u)||_2$ where $u\in W$ is a given field and $v,w$ are any two fields in $W$. Following the same steps used in the estimation of $||P(u) - P(v)||$ together with the bounds on the derivatives of $F_1$, a straight forward calculation shows that
\begin{align}\label{eq:bd Q 1}
||Q(v;u) - Q(w;u)|| \leq \frac{L_1}{\epsilon^2}||v-w|| \leq \frac{L_1}{\epsilon^2}||v-w||_2,
\end{align}
where $L_1 = 4 C_2 \bar{J}_1$. 

We now examine $||\nabla Q(v;u) - \nabla Q(w;u)||$. Taking gradient of (\ref{eq:Q simple}) we get
\begin{align}\label{eq:grad Q}
\nabla Q(v;u)(x) &= \dfrac{2}{\epsilon \omega_d} \int_{H_1(0)} \bar{\omega}_\xi(x) J(|\xi|) F''_1(\sqrt{s_\xi} S_\xi(u)) e_\xi \otimes \nabla S_\xi(v) d\xi \notag \\
&+ \dfrac{2}{\epsilon \omega_d} \int_{H_1(0)} J(|\xi|) F''_1(\sqrt{s_\xi} S_\xi(u)) S_\xi(v) e_\xi \otimes \nabla \bar{\omega}_\xi(x) d\xi \notag \\
&+ \dfrac{2}{\epsilon \omega_d} \int_{H_1(0)} \bar{\omega}_\xi(x) J(|\xi|) \sqrt{s_\xi} F'''_1(\sqrt{s_\xi} S_\xi(u)) S_\xi(v) e_\xi \otimes \nabla S_\xi(u) d\xi \notag \\
&=: G_1(v;u)(x) + G_2(v;u)(x) + G_3(v;u)(x).
\end{align}
It is straight forward to show that
\begin{align*}
||G_1(v;u) - G_1(w;u)|| &\leq  \frac{L_1}{\epsilon^2} ||\nabla v- \nabla w|| \leq \frac{L_1}{\epsilon^2} ||v - w||_2 \notag \\
||G_2(v;u) - G_2(w;u)|| &\leq  \frac{L_1 C_{\omega_1}}{\epsilon^2} || v- w|| \leq \frac{L_1 C_{\omega_1}}{\epsilon^2} || v- w||_2.
\end{align*}
Applying the inequalities $|S_\xi(v) - S_\xi(w)| \leq 2||v-w||_\infty /(\epsilon |\xi|) \leq 2 C_{e_1} ||v-w||_2 /(\epsilon|\xi|)$ and $|F'''_1(r)| \leq C_3$, we have
\begin{align*}
|G_3(v;u)(x) - G_3(w;u)(x)| &\leq \frac{4C_3 C_{e_1}||v-w||_2}{\epsilon^{5/2}} \frac{1}{\omega_d}\int_{H_1(0)} \frac{J(|\xi|)}{|\xi|^{3/2}} \epsilon |\xi| |\nabla S_\xi(u)| d\xi.
\end{align*}
Using the estimates above one has
\begin{align*}
||G_3(v;u) - G_3(w;u)|| &\leq \frac{8C_3 C_{e_1} \bar{J}_{3/2}}{\epsilon^{5/2}} ||u||_2 ||v-w||_2 = \frac{L_2}{\epsilon^{5/2}} ||u||_2 ||v-w||_2,
\end{align*}
where $L_2 = 8 C_3 C_{e_1}\bar{J}_{3/2}$. On collecting results, we have shown
\begin{align}\label{eq:bd Q 2}
||\nabla Q(v;u) - \nabla Q(w;u)|| &\leq \frac{\sqrt{\epsilon} L_1 (1+C_{\omega_1}) + L_2 ||u||_2}{\epsilon^{5/2}} ||v-w||_2.
\end{align}

Next we take the gradient of (\ref{eq:grad Q}), and write
\begin{align}\label{eq:grad2 Q}
\nabla^2 Q(v;u)(x) &= \nabla G_1(v;u)(x) + \nabla G_2(v;u)(x) + \nabla G_3(v;u)(x). 
\end{align}
Following the steps used in previous subsection, we estimate each term in  (\ref{eq:grad2 Q}) to obtain the following estimate given by
\begin{align}\label{eq:bd Q 3}
||\nabla^2 Q(v;u) - \nabla^2 Q(w;u)|| &\leq \frac{L_5(1 + ||u||_2 + ||u||_2^2)}{\epsilon^3}||v-w||_2. 
\end{align}
The proof of Theorem \ref{thm:Q lip in H2} is completed on summing up (\ref{eq:bd Q 1}), (\ref{eq:bd Q 2}), (\ref{eq:bd Q 3}) under the hypothesis $\epsilon<1$.

\section{Discussion of bounds for the rate of convergence}\label{s:discussion}

In this section we illustrate the time scales involved in a fracture simulation and provide an example where the error given by \eqref{eq:err estimate 1} can be controlled with acceptable computational complexity. We consider a  one meter cube with horizon $\epsilon =1/10$. Let the shear wave speed be given by  $\nu = 1400 \; \text{meter}/\text{second}$. This wave speed is characteristic of Plexiglas. The characteristic time $T^*$ is the time for a shear wave to travel across the cube and is $718 \; \mu$-seconds. Let $T$ be the non-dimensional simulation time and the physical time of the simulation is $TT^* \;\mu$-seconds. From Theorem \ref{thm:convergence} we have 
\begin{align*}
\sup_{1\leq k \leq T/\Delta t } E^k \leq C_p h^2 + \exp[T(1+L)/\epsilon^2] \left( e^0 + C_t T\Delta t + C_s T h^2/\epsilon^2 \right)\frac{1}{(1-\Delta t)^2}
\end{align*}
where
\begin{align*}
L_1 &= 4 (\sup_r |F_1''(r)|) \frac{1}{\omega_d} \int_{H_1(0)} J(|\xi|)/|\xi| d\xi, \\
C_p &= c \sup_t ||u(t)||_2 + c \sup_t ||v||_2, \\
C_t &= \sup_t ||\frac{\partial^2 u}{\partial t^2}|| + \sup_t ||\frac{\partial^2 v}{\partial t^2}||, \\
C_s &= L_1 c \sup_t ||u||_2.
\end{align*}
Ignoring the error in initial data $e^0$, and the $C_p$ term, we set $\bar{L} = 1 + L_1 = 1 + 4$ and we get
\begin{align}\label{eq:err estimate 1}
\sup_{1\leq k \leq T/\Delta t } E^k \leq \exp[5T/\epsilon^2] \left( C_t T\Delta t + (5T/\epsilon^2) h^2 \sup_t ||u||_2  \right)\frac{1}{(1-\Delta t)^2}.
\end{align}
One observes that the primary source of error in  \eqref{eq:err estimate 1} is the second term
\begin{align}\label{eq:err estimate 2}
 \exp[5T/\epsilon^2]  (5T/\epsilon^2) h^2 \sup_t ||u||_2 .
\end{align}
We choose $T$ so that $8=5T/\epsilon^2$ and $(5T/\epsilon^2)\exp[5T/\epsilon^2]=23,850$. Hence $T=0.016$ and the physical simulation time is $11.48\mu$ sec. If we choose $h=0.00142$ calculation shows the relative error associated with this term is $0.05$ with $34$ million spatial degrees of freedom. Here degrees of freedom are given by $\epsilon^{-3}\times h^{-3}$. The solution after cycle time T can be used as the initial conditions for a subsequent run and the process can be iterated. These estimates predict a total physical simulation time of $114.8\mu$-second before the factor multiplying $\sup_t ||u||_2$ in \eqref{eq:err estimate 2} becomes greater than $1/10$. This time span is of the same order as physical experiments.
As this calculation is based upon a-priori error estimates it is naturally pessimistic.

\section{Conclusions}\label{s:concl}
We have considered a canonical nonlinear peridynamic model and have shown the existence of a unique $H^2(D;\R^d) \cap H^1_0(D;\R^d)$ solution for any finite time interval. 
It has been demonstrated that finite element approximation converges to the exact solution at the rate $C_t\Delta t + C_s\frac{h^2}{\epsilon^2}$ for fixed $\epsilon$. The constants $C_t$ and $C_s$ are independent of time step $\Delta t$ and mesh size $h$. The constant $C_t$ depends on the $L^2$ norm of the first and second time derivatives of the solution. The constant $C_s$ depends on the $H^2$ norm of the solution. A  stability condition for the length of time step has been obtained for the linearized peridynamic model. It is expected that this stability condition is also in force for the nonlinear model in regions where the deformation is sufficiently small.

We remark that there is a large amount of work regarding asymptotically compatible schemes, in which one can pass to the  limit $\epsilon \to 0$ and retain convergence of the numerical method, see \citet{CMPer-Tian3,ChenBakenhusBobaru}. Such a scheme may be contemplated only when the convergence rate to the solution of the limit problem with respect to $\epsilon$ is known. 
Unfortunately an asymptotically compatible scheme is not yet possible for the nonlinear nonlocal evolutions treated here because the convergence rate of solutions with respect to $\epsilon>0$ is not known. One fundamental barrier to obtaining a rate is that the complete characterization of the $\epsilon=0$ evolution is not yet known. What is known so far is the characterization developed in the earlier work \cite{CMPer-Lipton3, CMPer-Lipton}. Here the evolution $u^{\epsilon}$ for the nonlinear nonlocal model is shown, on passage to subsequences, to converge in the $C([0,T];L^2(D;\R^3))$ norm as $\epsilon \rightarrow 0$ to a limit evolution $u(t)$ the space of $SBD^2$ functions. The fracture set at time $t$ is given by the jump set $J_{u(t)}$ of $u(t)$.  ($J_{u(t)}$ is the countable union of components contained in smooth manifolds and has finite Hausdorff $d-1$ measure.) At each time the associated energies $PD^\epsilon(u^{\epsilon})$ $\Gamma$-converge to the energy of linear elastic fracture mechanics evaluated at the limit evolution $u(t)$. This energy is found to be bounded. Away from the fracture set the limit evolution has been shown to evolve according to the linear elastic wave equation \cite{CMPer-Lipton3, CMPer-Lipton}.  What remains missing is the dynamics of the fracture set $J_{u(t)}$. Once this is known a convergence rate may be sought and an asymptoticly compatible scheme may be  contemplated. 

As shown in this paper the nonlinear nonlocal model is well posed in $H^2$ for all  $\epsilon >0$. However the $H^2$ norm of the solution gets progressively larger as $\epsilon\rightarrow 0$ if gradients steepen due to forces acting on the body. 
On the other hand if it is known that the solution is bounded in a $C^p$ norm uniformly for $\epsilon >0$ and if $p$ large enough then one can devise a finite difference scheme with truncation error that goes to zero independent of the peridynamic horizon [Proposition 5, \cite{CMPer-JhaLipton2}]. For example if $p\geq 4$  then the peridynamic evolutions converge to the elastodynamics evolution independently of horizon and an asymptotically compatible scheme can be developed for the linearized peridynamic force, [Proposition 5, \cite{CMPer-JhaLipton2}]. We note that the nonlinear and linearized kernels treated here and in  \cite{CMPer-JhaLipton2} are  different than those treated in  \cite{CMPer-Tian3} where asymptotically compatible schemes are first proposed. 

Last we believe that it is possible to show that the time stepping methods developed here provide approximation in the $H^1$ norm with a spatial convergence rate of $h/\epsilon^{2+\alpha}$, where $\alpha>0$ is to be determined. The calculation is anticipated to be tedious but we believe that the convergence rate can be established using the techniques established in this paper.

%
%
%



\end{document}